\documentclass[11pt,letterpaper]{article}

\usepackage[margin=1in]{geometry}
\usepackage[T1]{fontenc}
\usepackage{amsmath,amsfonts,amssymb,amscd,amsthm}
\usepackage{mathrsfs}
\usepackage{graphicx}
\usepackage{empheq}
\usepackage{multirow}
\usepackage{adjustbox,booktabs}
\usepackage{algorithm}
\usepackage{algpseudocode}
\usepackage[colorlinks=true,linkcolor=blue,urlcolor=blue,citecolor=red]{hyperref}

\allowdisplaybreaks

\newtheorem{theorem}{Theorem}[section]

\newtheorem{proposition}[theorem]{Proposition}

\theoremstyle{definition}
\newtheorem{definition}[theorem]{Definition}

\newcommand{\ep}{\varepsilon}

\newcommand{\vw}{\mathbf{w}}
\newcommand{\vx}{\mathbf{x}}

\newcommand{\cK}{\mathcal{K}}
\newcommand{\cL}{\mathcal{L}}

\newcommand{\cO}{\mathcal{O}}
\newcommand{\cP}{\mathcal{P}}

\newcommand{\R}{\mathbb{R}}

\newcommand{\N}{\mathbb{N}}

\newcommand{\I}{\mathbb{I}}

\newcommand{\al}{\alpha}
\newcommand{\be}{\beta}

\newcommand{\lam}{\lambda}

\newcommand{\Om}{\Omega}

\DeclareMathOperator{\prox}{prox}
\DeclareMathOperator*{\argmin}{argmin}

\newcommand{\norm}[1]{\left\lVert#1\right\rVert}
\newcommand{\bnorm}[1]{\left\| #1 \right\|_{\bullet}}

\newcommand{\ti}{\times}

\newcommand{\rank}{\text{rank}}

\newcommand{\la}{\langle}
\newcommand{\ra}{\rangle}

\newcommand{\stt}{\quad \text{s.t.} \quad}

\newcommand{\sign}{\text{sign}}

\newcommand{\mzero}{\text{O}}

\title{Separable Nonnegative Matrix Factorization Using Powered Ratio-of-Norms Regularization}

\author{Matthew McCarver \quad Jing Qin\thanks{Corresponding author. The research of Qin is supported by NSF grant DMS-2607997.}\\
\small Department of Mathematics, University of Kentucky, Lexington, KY 40506, USA}

\date{}

\begin{document}
\maketitle

\begin{abstract}
Separable nonnegative matrix factorization (SNMF) has been widely used for low-rank representation and clustering of nonnegative data, owing to its ability to produce part-based and interpretable decompositions. In particular, SNMF is closely related to graph clustering and community detection. To enhance sparsity and identifiability of the learned factors, we propose an $\ell_1^p/\ell_2$-regularized SNMF model based on a powered ratio-of-norms regularizer. The resulting formulation is nonconvex and nonsmooth, which poses significant challenges for optimization. To address this, we develop efficient algorithms based on the difference-of-convex function algorithm (DCA) and the alternating direction method of multipliers (ADMM). The proposed methods decompose the original problem into tractable subproblems, leveraging closed-form proximal operators associated with the powered norm terms. We establish descent and limiting criticality properties for the DCA scheme and convergence under standard assumptions for the ADMM scheme. Extensive numerical experiments on synthetic datasets and hand gesture classification tasks demonstrate that the proposed approach achieves competitive or improved performance in anchor identification and classification accuracy compared with existing SNMF methods, while maintaining competitive computational efficiency.
\end{abstract}



\section{Introduction}

Nonnegative Matrix Factorization (NMF) has been developed as an efficient tool for low-rank approximation with explicit nonnegativity constraints \cite{Paatero1994, Lee1999}. Specifically, NMF approximates a nonnegative data matrix as the product of two nonnegative factors. In contrast to the singular value decomposition (SVD), this formulation yields an interpretable, parts-based representation of the data. As a result, NMF has been widely adopted in applications such as text mining, computational biology, hyperspectral unmixing, and computer vision. One particularly important variant of NMF is Separable Nonnegative Matrix Factorization (SNMF), which assumes that the data lie in a convex cone generated by a subset of its own columns, which are often referred to as anchor points or extreme rays. Under this separability condition, the factorization reduces to identifying these representative columns from the data matrix itself, leading to computationally efficient and provably tractable convex and combinatorial algorithms. This structure underpins its effectiveness in applications such as hyperspectral imaging, topic modeling, and blind source separation.
The SNMF model assumes that the basis matrix is formed by a subset of columns of the data matrix. Consequently, the central task is to identify the corresponding index set, which represents the extreme points of the data geometry. This motivates the use of regularization techniques that promote sparsity and low-rankness of the underlying matrix and enable reliable column selection.

Classical convex regularizers, such as $\ell_1$ norm and nuclear norm, encourage sparsity or low-rank structure, but they are generally \emph{scale-sensitive}. As a result, they may fail to distinguish extreme columns from interior points when the data possess varying magnitudes. In the SNMF setting, the identification of anchor columns should depend primarily on their relative geometric importance rather than their absolute scale.

To address this limitation, ratio-based regularization has been proposed. The $\ell_1/\ell_2$ regularization, i.e., the ratio of the $\ell_1$-norm and the $\ell_2$-norm, has been applied in sparse signal recovery.
First introduced in the context of NMF~\cite{Hoyer2002}, this regularization has attracted increasing attention in the compressed sensing community for promoting sparse nonnegative signals, largely due to its scale-invariant property. More recently, it has been shown that the $\ell_0$ regularizer is equivalent to the $\ell_1/\ell_2$ ratio for nonnegative signals~\cite{Yin2015}.
Solving the $\ell_1/\ell_2$ model using ADMM was the standard practice \cite{Rahimi2019, Tao2022, Wang2022}. In addition, a proximal operator was introduced for solving the $\ell_1/\ell_2$ model in conjunction with ADMM \cite{Tao2022}. To enhance computational efficiency, an accelerated scheme has been developed to solve the $\ell_1/\ell_2$ model based on the DCA \cite{Wang2021}. In addition, several extensions have been introduced to further enhance sparsity. For instance, difference-based regularizers of the form $\alpha \ell_1 - \beta \ell_2$ \cite{Ding2019} and quasi-norm ratio $\ell_{1/2}/\ell_2$ \cite{Yu2025} have been explored. Empirically, $\ell_1/\ell_2$-type models have demonstrated strong performance across a range of applications, including CT reconstruction, blind deconvolution, sparse portfolio optimization, and image gradient sparsification~\cite{Wang2021,Esser2015a,Esser2015,Wu2023,Wang2022}.

Extending this idea to matrices, the ratio of the nuclear norm to the Frobenius norm has been studied in the context of low-rank matrix recovery, where local convergence guarantees have been established~\cite{Gao2024}. Motivated by this success, we have adapted this regularizer for solving the SNMF problem, with applications to hand gesture classification~\cite{McCarver2024}.

Motivated by the recent study of $\ell_1^p$ regularizer and its outstanding performance in tensor recovery \cite{Henneberger2024} and hyperspectral band selection \cite{henneberger2025hyperspectral}, we propose a class of SNMF models based on the powered ratio-of-norms $\ell_1^p/\ell_2$  regularizer, defined as the ratio of the $p$-th power of a matrix norm to the Frobenius norm, to enhance sparsity and low-rankness. For the matrix norm in the numerator, we consider two specific cases: matrix entrywise $\ell_1$-norm and nuclear norm throughout this paper. To solve the resulting models, we develop DCA- and ADMM-based algorithms. In the numerical experiments, the DCA scheme is used for the standard ratio case $p=1$, while the ADMM scheme enables a systematic investigation of $p\in\{1,2,3,4\}$. We demonstrate their performance on synthetic datasets and compare them with other related SNMF approaches in terms of identification accuracy and runtime. We also present real-world experiments on hand gesture classification. Finally, we provide a comprehensive analysis, including ablation studies, convergence behavior across different data regimes, computational complexity comparisons, and performance under high-noise conditions.

The remainder of the paper is organized as follows. In Section~\ref{sec:pre}, we introduce a powered ratio-of-norms regularizer and its proximal operator. Section~\ref{sec:method} introduces a class of SNMF models based on powered ratio-of-norms regularization, along with two algorithms based on DCA and ADMM for solving the resulting nonconvex problem. In Section~\ref{sec:exp}, we evaluate the proposed method on synthetic datasets and hand gesture classification tasks, and demonstrate its effectiveness through extensive comparisons with existing SNMF approaches. A comprehensive discussion including ablation studies and computational complexity analysis is provided in Section~\ref{sec:dis}. Finally, Section~\ref{sec:con} concludes the paper and outlines directions for future work.

\section{Preliminaries}\label{sec:pre}

We summarize the notation used throughout this work. Vectors are denoted by bold lowercase letters, e.g., $\vx$, and matrices by uppercase letters, e.g., $X$. The zero vector and zero matrix are denoted by $\mathbf{0}$ and $\mzero$, respectively; the $n\times n$ identity matrix is denoted by $I_n$. The $\ell_1$-norm $\|\cdot\|_1$ denotes the sum of the absolute values of the entries, while $\|\cdot\|_2$ denotes the Euclidean norm for vectors. For matrices, $\|\cdot\|_F$ denotes the Frobenius norm, defined as the square root of the sum of squared entries, and $\|\cdot\|_*$ denotes the nuclear norm, i.e., the sum of the singular values of the matrix. We use $\R$ to denote the set of real numbers, $\N$ the set of natural numbers, $\R^n$ the set of all $n$-dimensional real-valued vectors, and $\R^{m\times n}$ the set of all real-valued matrices of size $m\times n$. We use $\Omega$ to denote the feasible set and $\mathcal{K}$ to denote the index set of selected columns.

In the SNMF problem, we seek a matrix $X \in \R_+^{n \times n}$ such that $M\approx MX$, where $X$ is sparse and ideally has only $r$ nonzero rows corresponding to the  $r$ selected columns of $M$. Under exact separability, these selected rows encode the anchor columns that generate the remaining columns of $M$. Directly enforcing such combinatorial structure is NP-hard in general, so we instead incorporate regularizers into the objective that promote the desired properties. The regularizers we employ fall into two categories: those promoting \emph{sparsity}, encouraging most entries of $X$ to be zero, and those promoting \emph{low-rankness}, encouraging $X$ to have few nonzero singular values. Both properties are natural for SNMF: under exact separability, the
coefficient matrix $X$ has only a small number of active rows, with the
anchor columns corresponding to standard basis vectors. Such convex regularizers include the matrix entrywise $\ell_1$ norm and nuclear norm.

While convex regularizers provide computational efficiency and theoretical guarantees, they may not fully capture the structural properties of interest, such as sparsity and low-rankness. This motivates the development of nonconvex approaches, among which ratio-of-norms regularizers offer a flexible and effective alternative.
\begin{definition}[Powered Ratio-of-Norms]\label{intro-def:RatioNorm-p}
    Let $X \in \R^{m\times n}\setminus\{\mzero\}$. The $\ell_1^p/\ell_2$ powered ratio-of-norms regularizer of $X$ is defined as
    \begin{equation}\label{intro:Ratio-p}
       R_{p}(X) := \frac{\|X\|_{\bullet}^p}{\|X\|_F}, \quad p\geq1,
    \end{equation}
    where $\|\cdot\|_{\bullet}$ is a general matrix norm. In particular, when $\|\cdot\|_{\bullet}$ is the matrix entrywise $\ell_1$ norm, i.e., the sum of absolute values of all matrix entries, we have
    $$
    R_{p,1}(X) := \frac{\|X\|_1^p}{\|X\|_F}, \quad p \geq 1.
    $$
     When $\|\cdot\|_{\bullet}$ is the matrix nuclear norm, denoted by $ \|\cdot\|_*$, i.e., the sum of all the singular values of the matrix, we have
    $$
    R_{p,*}(X) := \frac{\|X\|^p_*}{\|X\|_F}, \quad p\geq 1.
    $$
\end{definition}

The above two special types of powered ratio-of-norms regularizers share a common ratio-based nonconvex structure but promote different forms of low-complexity solutions. Specifically, $R_{p,1}$ encourages entrywise sparsity consistent with the separability assumption in SNMF, whereas $R_{p,*}$ promotes low-rank structure through spectral regularization. The following proposition shows that these two regularizers are comparable and equivalent up to dimension-dependent constants.

\begin{proposition}\label{prop:R}
    Let $X\in\R^{m\times n}\setminus\{\mzero\}$. The following properties about the $R_{p,1}(X)$ and  $R_{p,*}(X)$ hold:
    \begin{enumerate}
        \item Ordering: \begin{equation}\norm{X}_F^{p-1}\leq R_{p,*}(X)\leq R_{p,1}(X).\end{equation}
        When $\norm{X}_F=1$, we have $R_{p,*}(X)=1$ if and only if $\rank(X)=1$, and $R_{p,1}(X)=1$ if and only if $X$ has exactly one nonzero entry.

        \item Scale dependence: $R_p(cX)=|c|^{p-1}R_p(X)$. Thus, $R_{p}(X)$ is scale-invariant if $p=1$ and is positively homogeneous of degree $p-1$ if $p>1$.

        \item Structure-dependent bounds: (a) If $\rank(X)=r$, then
        \begin{equation}
        R_{p,*}(X)\leq r^{p/2}\norm{X}_F^{p-1}.
        \end{equation}
        The equality holds when the $r$ nonzero singular values are the same. \\
        (b) If $X\in\R^{m\times n}$ has at most $s$ nonzero rows, then
        \begin{equation}
        R_{p,1}(X)\le (sn)^{p/2}\norm{X}_F^{p-1}.
        \end{equation}
        In particular, since $s\leq m$ and $R_{p,*}(X)\ge \norm{X}_F^{p-1}$, we have
        \begin{equation}
            R_{p,1}(X)\leq (mn)^{p/2}R_{p,*}(X).
        \end{equation}
    \end{enumerate}
\end{proposition}

Figure~\ref{fig:ratio-geometry} illustrates the geometry of the powered ratio regularizer in two dimensions. The level sets retain a sparsity-promoting geometry concentrated along the coordinate axes, while their shapes vary with $p$.

\begin{figure}[ht!]
\centering
\begin{tabular}{cc}
    \includegraphics[width=.48\textwidth]{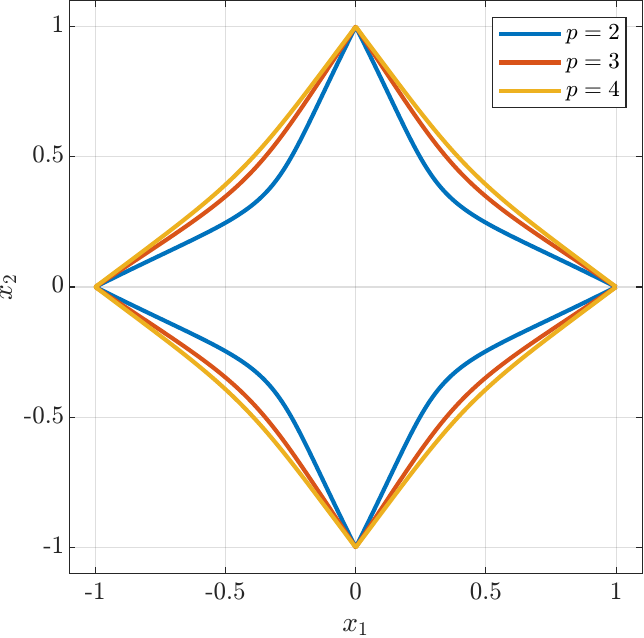}&
    \includegraphics[width=.48\textwidth]{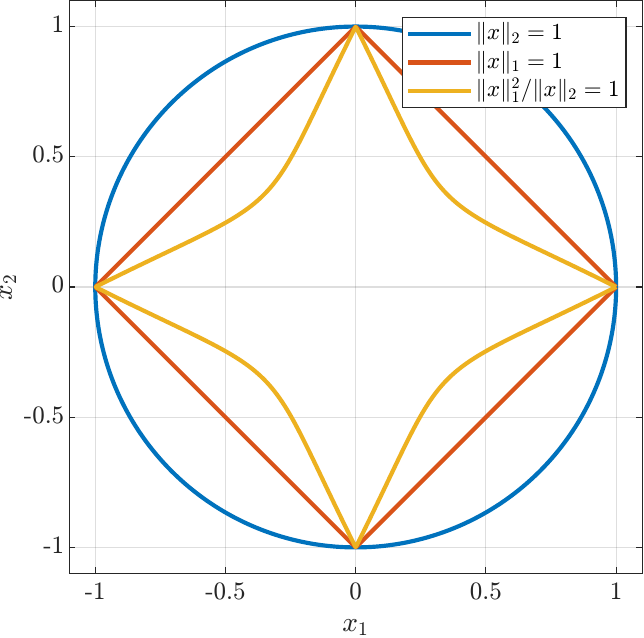}\\
    (a) Varying $p$ & (b) Comparison with standard norms
\end{tabular}
\caption{Geometry of the regularizer $R_p(x)=\|x\|_1^p/\|x\|_2$ in $\mathbb{R}^2$.}
\label{fig:ratio-geometry}
\end{figure}

Next we discuss the semi-algebraicity of $R_\rho$ which will be used for convergence analysis of the proposed algorithms.
\begin{definition}[Semi-algebraic set and function on matrix space]
A subset $\mathcal{S} \subseteq \mathbb{R}^{m\times n}$ is called \emph{semi-algebraic} if it can be represented as a finite union of sets of the form
\[
\mathcal{S} = \bigcup_{i=1}^N \bigcap_{j=1}^{M_i} \{ X \in \mathbb{R}^{m\times n} : f_{ij}(X) = 0,\; g_{ij}(X) > 0 \},
\]
where $f_{ij}, g_{ij} : \mathbb{R}^{m\times n} \to \mathbb{R}$ are polynomial functions in the entries of $X$.

A function $f:\mathbb{R}^{m\times n} \to \mathbb{R}$ is called \emph{semi-algebraic} if its graph
\[
\{(X,t)\in \mathbb{R}^{m\times n} \times \mathbb{R} : f(X)=t\}
\]
is a semi-algebraic set.
\end{definition}

\begin{theorem}[Semi-algebraicity of $R_p$]
\label{thm:semialg}
For any integer $p\geq 1$, the regularizer $R_p$ defined in
\eqref{intro:Ratio-p}, with $\norm{\cdot}_\bullet$ being either the matrix entrywise $\ell_1$-norm or the nuclear norm $\norm{\cdot}_*$, is
semi-algebraic on $\mathbb{R}^{m\times n}\setminus\{O\}$.
\end{theorem}

\begin{proof}
The entry-wise $\ell_1$ norm $\|X\|_1 = \sum_{i,j} |X_{ij}|$ is semi-algebraic since it is a finite sum of absolute values, and $|t| = \sqrt{t^2}$ is semi-algebraic. Hence $\|X\|_1^p$ is semi-algebraic for any integer $p \ge 1$. The nuclear norm $\|X\|_*$ is also semi-algebraic, since the singular values of $X$ are semi-algebraic functions of its entries; so $\|X\|_*^p$ is semi-algebraic for any integer $p\geq 1$.

Similarly, the Frobenius norm $\|X\|_F = \sqrt{\sum_{i,j} X_{ij}^2}$ is semi-algebraic. For $X \neq \mzero$, the function $R_p(X)$ is the ratio of two semi-algebraic functions with a nonvanishing denominator, and is therefore semi-algebraic.

\end{proof}

\begin{definition}
    The proximal operator of a proper lower semicontinuous function $R:\mathbb{R}^{m\times n}\to(-\infty,+\infty]$ is defined as
    \[
\prox_{\lambda R}(X)\in \argmin_{Y\in\R^{m\times n}}\left\{\lambda R(Y)+\frac{1}2\norm{X-Y}_F^2\right\},\quad \lambda>0.
    \]
\end{definition}
Note that solving the $\ell_1^p/\ell_2$-regularized problem requires evaluating the proximal operator of $\ell_1^p$. Since a closed-form expression for the proximal operator of $\ell_1^p$ is only available for $p \in \{1,2,3,4\}$~\cite{Henneberger2024}, we restrict our focus to these cases in this paper. For $p>4$, the proximal operator generally requires solving a scalar nonlinear equation numerically, e.g., by Newton's method. Raising the numerator to a power $p > 1$ changes the penalty landscape and can increase the separation between desirable and undesirable solutions. However, as shown in Proposition~\ref{prop:R}, scale invariance holds only when $p=1$.

We next summarize the proximal operators studied in~\cite{Henneberger2024} that are used in our algorithms.

\begin{theorem}[Proximal Operator for $\ell_1^p$ \cite{Henneberger2024}]\label{intro-def:ProxOpL1p}
    The proximal operator of $\ell_1^p$ with $p\in\{1,2,3,4\}$ is given by
    \begin{equation}\label{intro:L1p-proxOp}
        \prox_{\lam \|\cdot\|_1^p}(A) = \sign(A) \odot \max\{|A|-p\lam g_p^{p-1}, 0\},
    \end{equation}
    where $\odot$ is entry-wise multiplication and $|A|$ is entry-wise absolute value. The constant $g_p$ which depends on $p$ is given by
    \begin{equation}\label{intro:gpConst}
        g_p = \begin{cases}
            1, & \quad p =1; \\
            \frac{\|A\|_1}{2k\lam + 1}, & \quad p =2; \\
            \frac{-1 + \sqrt{1 + 12k\lam \|A\|_1}}{6k\lam}, & \quad p =3;  \\
            \sqrt[3]{\frac{1}{8k\lam}(\|A\|_1 + \sqrt{\Delta})} + \sqrt[3]{\frac{1}{8k\lam}(\|A\|_1 - \sqrt{\Delta})}, & \quad p =4,
        \end{cases}
    \end{equation}
    where $k = m\ti n$ and $\Delta = \|A\|_1^2 + \frac{1}{27 k \lam}$.
\end{theorem}
We note that when $p=1$, the constant $g_1 = 1$ and the operator \eqref{intro:L1p-proxOp} reduces to the classical soft thresholding operator. Moreover, this result extends naturally to the proximal operator of the $p$th power of the matrix nuclear norm \cite{Henneberger2024}. Let $A = U \Sigma V^T$ be the Singular Value Decomposition (SVD) of $A$. Then
\begin{equation}
\prox_{\lambda \|\cdot\|_*^p}(A) = U \prox_{\lambda \|\cdot\|_1^p}(\Sigma) V^T.
\end{equation}
This formulation generalizes the singular value thresholding (SVT) operator \cite{Cai2010}, which corresponds to the special case $p = 1$.

\section{Proposed Methods}\label{sec:method}
To promote sparsity and low-rankness, we propose a class of SNMF models based on the powered ratio-of-norms $\ell_1^p/\ell_2$  regularization
\begin{equation}\label{ch3:RatioModel}
    \min_{X\in \Om} \left\{ \lam R_p(X) + \frac12 \|MX-M\|_F^2 \right\} ,
\end{equation}
where $M\neq\mzero$, $R_p(X)$ is given in Definition~\ref{intro-def:RatioNorm-p} and
\begin{equation}\label{ch3:OmDef}
\Om =\left\{X\in \R^{n\times n}_+ \;\big|\; x_{ii} \leq 1,\; w_i x_{ij} \leq w_j x_{ii},\, i,j=1,2,\ldots,n\right\}.
\end{equation}
Let $\I_{\Omega}(X)$ denote the indicator function of $\Omega$ defined as
\[
\I_{\Omega}(X)=\left\{
\begin{aligned}
    &0,&&\mbox{if }X\in \Omega;\\
    &\infty,&&\mbox{if }X\notin \Omega.
\end{aligned}
\right.
\]
Using the indicator function, the constrained problem can be converted to its unconstrained version:
\begin{equation}\label{eqn:uncon_model}
    \min_{X} \left\{ \lam R_p(X) + \frac12 \|MX-M\|_F^2 +\I_{\Omega}(X)\right\}.
\end{equation}
In what follows, we develop two algorithms based on DCA and ADMM.

\subsection{Powered Ratio-of-Norms via DCA}
We begin with an interpretation based on fractional programming that motivates the DCA formulation. Consider the constrained SNMF problem
\begin{equation}\label{ch3:ratio-constrained-model}
    \min_{X \in \mathcal{F}} R_p(X),
\end{equation}
where
\[
\mathcal{F}:=\{X\in\Om:\|MX-M\|_F^2\leq\sigma\}.
\]
Here $\sigma>0$ represents the noise level and is related to the regularization parameter $\lam$ in \eqref{eqn:uncon_model}. We assume that $\|X\|_F>0$ for all $X\in\mathcal{F}$, so $R_p(X)$ is well-defined.

By extending \cite[Proposition~1]{Wang2020}, we obtain the following standard fractional-programming characterization.
\begin{proposition}
Assume that $\mathcal F$ is nonempty and that the minimum in
\eqref{ch3:ratio-constrained-model} is attained. Let
\begin{equation}\label{ch3:alpha*}
    \al^*=\min_{X\in\mathcal{F}}R_p(X),
\end{equation}
and define
\begin{equation}\label{ch3:Talpha}
    T(\al)=\inf_{X\in\mathcal{F}}
    \left\{\|X\|_\bullet^p-\al\|X\|_F\right\}.
\end{equation}
Then $\alpha^*$ is the optimal value of \eqref{ch3:ratio-constrained-model} if and only if
$T(\alpha^*)=0$.
\end{proposition}
This equivalence is also a direct consequence of the classical parametric reformulation in fractional programming \cite{dinkelbach1967nonlinear}. We emphasize, however, that \eqref{ch3:ratio-constrained-model} is used here only to motivate the parametric DC subproblem employed in Algorithm~\ref{alg:RatioDCA}; it is not the optimization problem solved directly by the algorithm.

For the penalized model~\eqref{eqn:uncon_model}, we set
\[
    \al_k=R_p(X_k)
\]
and, at the $k$-th outer iteration, consider the parametric DC problem
\begin{equation}\label{eq:dca-parametric}
    \min_X\Phi_k(X):=g(X)-\al_k h(X),
\end{equation}
where, for the $\ell_1^p/\ell_2$ model,
\[
g(X)=\lam\|X\|_1^p+\frac12\|MX-M\|_F^2+\I_\Omega(X),
\qquad
h(X)=\|X\|_F.
\]
The nuclear-norm variant is obtained by replacing $\|X\|_1^p$ with $\|X\|_*^p$ throughout the same derivation.

Since $h$ is convex and differentiable at $X_k\neq\mzero$, its affine minorization at $X_k$ is
\[
h(X)\geq h(X_k)+
\left\langle\frac{X_k}{\|X_k\|_F},X-X_k\right\rangle.
\]
Dropping terms independent of $X$ and adding a proximal term with $\be>0$ gives the proximal DCA subproblem
\begin{equation}\label{ch3:NNDCA Model}
X_{k+1}\in\argmin_X
\left\{
g(X)-\al_k\left\langle X,\frac{X_k}{\|X_k\|_F}\right\rangle
+\frac{\be}{2}\|X-X_k\|_F^2
\right\}.
\end{equation}
Completing the square yields the outer subproblem used in Algorithm~\ref{alg:RatioDCA}:
\begin{subequations}
\begin{empheq}[left=\empheqlbrace]{align}
X_{k+1}&\in\argmin_X\left\{
g(X)+\frac{\be}{2}
\left\|X-X_k-\frac{\al_k}{\be}\frac{X_k}{\|X_k\|_F}\right\|_F^2
\right\},
\label{ch3:DCA-ObjFunOuter-X}\\
\al_{k+1}&=R_p(X_{k+1}).
\label{ch3:DCA-ObjFunOuter-alpha}
\end{empheq}
\end{subequations}

To solve \eqref{ch3:DCA-ObjFunOuter-X}, we introduce a copy $V$ for the nonsmooth constrained term and apply ADMM with scaled dual variable $Z$. The corresponding augmented Lagrangian is
\begin{align}\label{ch3:DCARatio-lag}
L_\rho(X,V;Z)
=&\lam\|V\|_1^p+\frac12\|MX-M\|_F^2
+\frac{\be}{2}\left\|X-X_k-
\frac{\al_k}{\be}\frac{X_k}{\|X_k\|_F}\right\|_F^2\notag\\
&+\frac{\rho}{2}\|X-V+Z\|_F^2+\I_\Omega(V).
\end{align}
At the $j$-th inner iteration, the ADMM updates are
\begin{equation}\label{ch3:DCARatio-ADMM}
\left\{
\begin{aligned}
X^{j+1}&=\argmin_X
\frac12\|MX-M\|_F^2
+\frac{\rho}{2}\|X-V^j+Z^j\|_F^2
+\frac{\be}{2}\left\|X-X_k-
\frac{\al_kX_k}{\be\|X_k\|_F}\right\|_F^2,\\
V^{j+1}&=\argmin_{V\in\Om}
\lam\|V\|_1^p+
\frac{\rho}{2}\|V-(X^{j+1}+Z^j)\|_F^2,\\
Z^{j+1}&=Z^j+X^{j+1}-V^{j+1}.
\end{aligned}
\right.
\end{equation}
Here $j$ is the inner iteration index and $k$ is the outer iteration index. The $X$-subproblem is a least-squares problem with the closed-form solution
\begin{equation}\label{ch3:DCARatio-ADMM-xupdt}
X^{j+1}=\left(M^TM+(\be+\rho)I_n\right)^{-1}
\left(M^TM+\be X_k+
\frac{\al_kX_k}{\|X_k\|_F}+\rho(V^j-Z^j)\right)
:=Q^{-1}W.
\end{equation}
Since $Q=M^TM+(\be+\rho)I_n$ is symmetric positive definite, we factor $Q=R^TR$ once and solve the resulting triangular systems by forward and backward substitution. After this factorization, each solve costs $\mathcal{O}(n^2)$ operations.

The exact $V$-update is the constrained proximal map
\begin{equation}\label{ch3:DCARatio-ADMM-vupdt-NN}
V^{j+1}=
\prox_{\frac{\lam}{\rho}\|\cdot\|_1^p+\I_\Omega}
(X^{j+1}+Z^j).
\end{equation}
In general, this constrained proximal map is not automatically equal to the composition of the unconstrained proximal map and the projection $\Pi_\Omega$. Thus, a projected-proximal implementation should be regarded as an inexact splitting realization unless that composition is established for the particular set $\Om$. The descent result below concerns the exact solution of \eqref{ch3:NNDCA Model}, equivalently an exact solution of its inner ADMM problem. In the numerical implementation, the exact constrained proximal update \eqref{ch3:DCARatio-ADMM-vupdt-NN} is approximated by applying the unconstrained proximal operator followed by projection onto $\Omega$.

For initialization, Algorithm~\ref{alg:RatioDCA} accepts an initial matrix $X_0\in \Omega$. Unless otherwise specified, we use $X_0=I_n$. Given $X_0$, the initial ratio parameter $\alpha_0$ is computed according to its definition from $X_0$. Once the outer iteration terminates, the column index set $\cK$ is obtained by selecting either the $r$ largest diagonal entries of $X$ or the $r$ rows of $X$ with the largest $\ell_2$-norms.

\begin{algorithm}[ht!]
\caption{$\ell_1^p/\ell_2$-Regularized SNMF via DCA (DCA)}
\label{alg:RatioDCA}
\begin{algorithmic}
\Require $M\in\R^{m\ti n}_+$, initial $X_0$, number $r$ of columns to extract, parameters $\rho,\be,\lam$, maximum iteration numbers $K$ and $J$ for DCA and ADMM, and tolerances $\ep$ and $\eta$ for DCA and ADMM.
\Ensure $X\in\R^{n\ti n}_+$ and column indices $\cK$.
\State $\al_0\gets R_p(X_0)$
\For{$k=0,\ldots,K-1$} \Comment{DCA loop}
    \State $X^0\gets X_k$, $V^0\gets X_k$, $Z^0\gets\mzero$
    \For{$j=0,\ldots,J-1$} \Comment{ADMM loop}
        \State Update $X^{j+1}$ via \eqref{ch3:DCARatio-ADMM-xupdt}
        \State Update $V^{j+1}$ via \eqref{ch3:DCARatio-ADMM-vupdt-NN}
        \State $Z^{j+1}\gets Z^j+X^{j+1}-V^{j+1}$
        \If{$\|X^{j+1}-X^j\|_F/\|X^j\|_F<\eta$}
            \State \textbf{break}
        \EndIf
    \EndFor
    \State $X_{k+1}\gets X^{j+1}$
    \State $\al_{k+1}\gets R_p(X_{k+1})$
    \If{$\|X_{k+1}-X_k\|_F/\|X_k\|_F\leq\ep$}
        \State \textbf{break}
    \EndIf
\EndFor
\State $X\gets X_{k+1}$
\State $\cK\gets\operatorname{post\text{-}processing}(X,r)$
\end{algorithmic}
\end{algorithm}

The coefficient $\al_k$ changes at every outer iteration, so Algorithm~\ref{alg:RatioDCA} is not ordinary DCA applied to one fixed objective. The following result therefore states the descent property for the actual parametric objective used at each iteration, without inferring monotonicity of a single objective sequence.

\begin{theorem}[Conditional convergence of Algorithm~\ref{alg:RatioDCA}]
\label{thm:dca-convergence}
Let $\{X_k\}$ be generated by Algorithm~\ref{alg:RatioDCA}. Assume that $X_k\neq\mzero$ and that each inner ADMM problem is solved exactly. Define
\[
\Phi_k(X)=g(X)-\al_k h(X),
\qquad
\al_k=R_p(X_k).
\]
Then
\begin{equation}\label{eq:dca-descent}
\Phi_k(X_{k+1})+
\frac{\be}{2}\|X_{k+1}-X_k\|_F^2
\leq\Phi_k(X_k).
\end{equation}
Moreover, $X_{k+1}$ satisfies
\[
0\in\partial g(X_{k+1})-
\al_k\frac{X_k}{\|X_k\|_F}
+\be(X_{k+1}-X_k).
\]
Furthermore, if
\[
X_k\to X^*\neq\mzero,
\qquad
\al_k\to\al^*,
\]
then
\[
0\in\partial g(X^*)-\al^*\partial h(X^*),
\]
which implies $X^*$ is a critical point of the limiting DC problem.
\end{theorem}

\begin{proof}
By convexity of $h$,
\[
h(X_{k+1})\geq h(X_k)+
\left\langle\nabla h(X_k),X_{k+1}-X_k\right\rangle.
\]
On the other hand, optimality of $X_{k+1}$ in \eqref{ch3:NNDCA Model}, evaluated at $X_k$, gives
\[
g(X_{k+1})-
\al_k\left\langle X_{k+1},\nabla h(X_k)\right\rangle
+\frac{\be}{2}\|X_{k+1}-X_k\|_F^2
\leq
g(X_k)-\al_k
\left\langle X_k,\nabla h(X_k)\right\rangle.
\]
Combining these inequalities yields \eqref{eq:dca-descent}. The first-order optimality condition of \eqref{ch3:NNDCA Model} gives
\[
0\in\partial g(X_{k+1})-
\al_k\nabla h(X_k)+
\be(X_{k+1}-X_k).
\]
Passing to the limit, using $\nabla h(X_k)=X_k/\|X_k\|_F$ and the closedness of $\partial g$, yields the stated criticality condition.
\end{proof}

Because $\al_{k+1}=R_p(X_{k+1})$, the function $\Phi_k$ changes with $k$. Therefore, \eqref{eq:dca-descent} does not by itself imply $\Phi_{k+1}(X_{k+1})\leq\Phi_k(X_k)$ or convergence of the entire sequence. Thus, the limiting criticality statement is conditional on the assumed convergence of $\{X^k\}$ and $\{\alpha_k\}$.

\subsection{Powered Ratio-of-Norms via ADMM}\label{ch3-sec:ADMMi}
We next develop an ADMM-based approach for solving the powered ratio-of-norms regularized model \eqref{ch3:RatioModel}, following a strategy similar to~\cite{Gao2024}. This provides a complementary optimization framework for handling the nonconvex regularization and the constraint set $\Omega$. In our implementation, the constraint is handled through an auxiliary variable and projection onto $\Omega$ defined in \eqref{ch3:OmDef}.

To apply the ADMM algorithmic framework, we introduce the auxiliary variable $Y$ and rewrite the unconstrained model \eqref{eqn:uncon_model} as
\begin{equation}\label{ch3:Ratio-ADMMP-modelAux}
    \min_{X,Y} \lam R_p(X) + \frac12 \|MY-M\|_F^2 + \I_{\Omega}(X) \stt X = Y.
\end{equation}
Then we define the augmented Lagrangian function as
\begin{equation}\label{ch3:Ratio-ADMMP-Lag}
    \cL(X,Y,U) = \lam R_p(X) + \frac12 \|MY-M\|_F^2 + \la U, X-Y\ra + \frac{\rho_1}{2}\|X-Y\|_F^2 + \I_{\Omega}(X).
\end{equation}
Here $U$ is the Lagrange multiplier and $\rho_1>0$ is the penalty parameter. Applying ADMM yields the following updates at the $k$-th iteration:
\begin{equation}\label{ch3:Ratio-ADMMi-steps}
\left\{\begin{aligned}
         X_{k+1} &= \argmin_{X} \cL(X,Y_k,U_k); \\
         Y_{k+1} &= \argmin_{Y}\cL(X_{k+1}, Y, U_k);  \\
         U_{k+1} &= U_k + \rho_1(X_{k+1}-Y_{k+1}), \\
    \end{aligned}\right.
\end{equation}
Note $k$ denotes the outer iteration index, used to distinguish it from the inner iterations introduced later.
This yields two subproblems: the $X$-subproblem involves the nonconvex powered ratio-of-norms regularizer, while the $Y$-subproblem is a least-squares problem about $Y$, which has a closed-form solution. We provide the detailed solutions to both subproblems in what follows.

\subsection{$X$-subproblem}
To solve the \(X\)-subproblem, we introduce two auxiliary variables, \(Z\) and \(W\), to decouple the numerator and denominator in the ratio term, as well as the indicator function:
\begin{equation}\label{ch3:Ratio-ADMMP-innerModel}
    \min_{X,Z,W} \lam \frac{\bnorm{X}^p}{\|Z\|_F} + \frac{\rho_1}{2}\|X-A^k\|_F^2 + \I_{\Omega}(W) \stt X = Z, X = W.
\end{equation}
Here $A^k=Y_k-U_k/\rho_1$.
Since the objective function in \eqref{ch3:Ratio-ADMMP-innerModel} is separable, we again apply the ADMM and define the augmented Lagrangian function as
\begin{equation}\label{ch3:Ratio-ADMMP-innerLag}
\begin{aligned}
        \cL_k(X,Z,W,V,S) &= \lambda\frac{\bnorm{X}^p}{\|Z\|_F} + \frac{\rho_1}{2}\|X-A^k\|_F^2 + \I_{\Omega}(W) + \la V, X-Z\ra  \\
        &+\frac{\rho_2}{2}\|X-Z\|_F^2 + \la S, X-W \ra + \frac{\rho_3}{2}\|X-W\|_F^2.
\end{aligned}
\end{equation}
Therefore, we obtain the following updates at the $j$th inner iteration
\begin{equation}\label{ch3:Ratio-ADMMP-stepsInner}
\left\{\begin{aligned}
         X^{j+1} &= \argmin_X \cL_k(X,Z^j,W^j,V^j,S^j); \\
         Z^{j+1} &= \argmin_Z \cL_k(X^{j+1},Z,W^j,V^j,S^j); \\
         W^{j+1} &= \argmin_W \cL_k(X^{j+1},Z^{j+1},W,V^j,S^j); \\
         V^{j+1} &= V^j + \rho_2(X^{j+1}-Z^{j+1}); \\
         S^{j+1} &= S^j + \rho_3(X^{j+1} - W^{j+1}) .
    \end{aligned}\right.
\end{equation}
By denoting $B^j = Z^j - {V^j}/{\rho_2}$, $D^j = W^j - S^j/\rho_3$ and assuming that $Z^j\neq \mzero$, the $X$-subproblem in \eqref{ch3:Ratio-ADMMP-stepsInner} yields
    \begin{align}
        X^{j+1} &=\argmin_X \frac{\lam}{\|Z^j\|_F}\bnorm{X}^p + \frac{\rho_1 + \rho_2 + \rho_3}{2}\left\| X - \frac{\rho_1 A^k + \rho_2 B^j + \rho_3 D^j}{\rho_1 + \rho_2 + \rho_3}\right\|_F^2 \notag\\
        &=\begin{cases}
            \prox_{\gamma \|\cdot\|_1^p}(\widetilde{C}), &  \text{if } R_{p}(X) = R_{p,1}(X); \label{ch3:Ratio-ADMMP-xupdt}\\
             \prox_{\gamma \|\cdot\|_*^p}(\widetilde{C}), & \text{if } R_{p}(X) = R_{p,*}(X).
        \end{cases}
    \end{align}
Here
$$
\widetilde{C} = \frac{\rho_1A^k + \rho_2B^j + \rho_3 D^j}{\rho_1 + \rho_2 + \rho_3}, \quad \gamma = \frac{\lam}{(\rho_1 + \rho_2 + \rho_3)\|Z^j\|_F}.
$$
For the $Z$-subproblem in \eqref{ch3:Ratio-ADMMP-stepsInner}, we follow the idea in \cite{Gao2024} and get
\begin{align}
    Z^{j+1}
        &= \argmin_{Z} \frac{\lam\bnorm{X^{j+1}}^p}{\|Z\|_F} + \frac{\rho_2}{2}\left\|X^{j+1} - Z + \frac{V^j}{\rho_2}\right\|_F^2 \notag\\
        &= \begin{cases}
        P^j, & \quad \text{ if }  C^j = \mzero; \label{ch3:Ratio-ADMMi-Zupdt}\\
        \zeta^j C^j, & \quad \text{ if } C^j \neq \mzero.
    \end{cases}
\end{align}
Here $C^j = X^{j+1} + \frac{V^j}{\rho_2}$, and $P^j$ is a random matrix satisfying $\|P^j\|_F = \sqrt[3]{{d^j}/{\rho_2}}$ and  $d^j = \lam \bnorm{X^{j+1}}^p$ and
$$
\zeta^j = \frac13 + \frac13(s^j + \frac{1}{s^j}), \quad s^j = \sqrt[3]{\frac{27q^j + 2 + \sqrt{(27q^j + 2)^2-4}}{2}}, \quad q^j = \frac{d^j}{\rho_2\|C^j\|_F^3}.
$$
The $W$-subproblem in \eqref{ch3:Ratio-ADMMP-innerModel} yields
\begin{equation}\label{ch3:Ratio-ADMMP-wupdt}
\begin{aligned}
    W^{j+1} &= \argmin_W \I_{\Omega}(W) + \frac{\rho_3}{2}\left\|W-\left(X^{j+1} + \frac{S^j}{\rho_3}\right)\right\|_F^2 \\
    &=\cP_{\Omega}\left(X^{j+1} + \frac{S^j}{\rho_3}\right).
\end{aligned}
\end{equation}

\subsection{$Y$-subproblem}
The solution to the $Y$-subproblem in \eqref{ch3:Ratio-ADMMi-steps} can be expressed as follows:
\begin{equation}\label{ch3:Ratio-ADMMi-yupdt}
\begin{aligned}
        Y_{k+1} &= \argmin_Y \frac12 \|MY-M\|_F^2 + \frac{\rho_1}{2}\left\|X_{k+1}- Y + \frac{U_k}{\rho_1}\right\|_F^2 \\
        &=(M^TM +\rho_1 I_n)^{-1}(M^TM + \rho_1 X_{k+1} + U_k).
\end{aligned}
\end{equation}
Similar to the treatment in \eqref{ch3:DCARatio-ADMM-xupdt}, we apply the Cholesky factorization to the matrix $M^TM+\rho_1I_n$, followed by forward and backward substitutions to enhance the computational efficiency.
Then once the solution $X$ is available, we extract the index set $\cK$ via post-processing. We extract $\mathcal{K}$ by selecting either the $r$ largest diagonal entries of $X$ or the $r$ rows having the largest $\ell_2$-norms.

The resulting ADMM-based algorithm is summarized in Algorithm \ref{alg:Ratio-ADMMP}.

\begin{algorithm}[ht!]
\caption{$\ell_1^p/\ell_2$-Regularized SNMF via projected ADMM (ADMM-P)}\label{alg:Ratio-ADMMP}
\begin{algorithmic}
\Require $M \in \R^{m\ti n}_+$, initial $X_0$, number $r$ of columns to extract, parameters $\rho_1,\rho_2, \rho_3, \lam$, maximum number of iterations $K$/$J$ for ADMM/ADMM-inner, and error tolerance $\ep$/$\eta$ for ADMM/ADMM-inner.
\Ensure $X \in \R^{n \ti n}_+$, column indices $\cK$.
\State $Y_0\gets X_0$, $U_0\gets \mzero$
\Comment{Initialize}
\For{$k =0:K-1$} \Comment{Outer ADMM Updates}
\State
$Z^0 \gets
\begin{cases}
X_0, & k=0,\\
Z^{j+1}, & k>0,
\end{cases}
\qquad
W^0 \gets
\begin{cases}
X_0, & k=0,\\
W^{j+1}, & k>0,
\end{cases}$
\Comment{Warm start}
\State $V^{0} \gets \mzero,\quad S^{0} \gets \mzero$ \Comment{Reset inner multipliers}
\For{$j=0: J-1$} \Comment{Inner ADMM Updates}
\State Update $X^{j+1}$ via \eqref{ch3:Ratio-ADMMP-xupdt}
\State Update $Z^{j+1}$ via \eqref{ch3:Ratio-ADMMi-Zupdt}
\State Update $W^{j+1}$ via \eqref{ch3:Ratio-ADMMP-wupdt}
\State Update $V^{j+1} = V^j + \rho_2(X^{j+1} -Z^{j+1})$ \Comment{Update Lagrange multiplier}
\State Update $S^{j+1} = S^j + \rho_3(X^{j+1}- W^{j+1})$
\If{$\|X^{j}-X^{j+1}\|_F/\|X^j\|_F < \eta$}
\State break
\EndIf
\EndFor
\State $X_{k+1} = X^{j+1}$ \Comment{Update $X$}
\State Update $Y_{k+1}$ via \eqref{ch3:Ratio-ADMMi-yupdt}
\State $U_{k+1} = U_k + \rho_1(X_{k+1}-Y_{k+1})$ \Comment{Update Lagrange multiplier}
\If{$\|X_{k}-X_{k+1}\|_F/\|X_k\|_F < \ep$} \Comment{Stopping criterion}
    \State break
\EndIf
\EndFor
\State $X = X_{k+1}$
\State $\cK = \text{post-processing}(X,r)$ \Comment{Post-processing to get $\cK$}
\end{algorithmic}
\end{algorithm}

The following convergence result applies to the idealized exact-subproblem realization of Algorithm~\ref{alg:Ratio-ADMMP}, and the projected and finitely terminated implementation used numerically is discussed after the theorem.

\begin{theorem}[Conditional convergence of Algorithm~\ref{alg:Ratio-ADMMP}]
\label{thm:admm-convergence}
Let $\{(X^k,Y^k,U^k)\}_{k\geq0}$ be the primal--dual sequence
generated by the idealized exact-subproblem counterpart of Algorithm~\ref{alg:Ratio-ADMMP}, and let
$\mathcal{L}$ be the augmented Lagrangian defined in
\eqref{ch3:Ratio-ADMMP-Lag}.
Assume that the sequence is bounded, all subproblems are solved
exactly, and the sufficient-descent and relative-error conditions
required by the standard nonconvex ADMM convergence framework
\cite{wang2019global} hold. Then every accumulation point is a
critical point of $\mathcal{L}$. Moreover, since $\mathcal{L}$
satisfies the Kurdyka--\L{}ojasiewicz (KL) property, the entire
primal--dual sequence converges to a critical point
$(X^*,Y^*,U^*)$ of $\mathcal{L}$. The limiting primal variables
satisfy the first-order stationarity conditions of the powered
ratio-of-norms regularized SNMF problem.
\end{theorem}

\begin{proof}
Under the assumed boundedness, sufficient-descent, and relative-error
conditions, the standard nonconvex ADMM convergence analysis
\cite{wang2019global} implies that every accumulation point of the
generated primal--dual sequence is a critical point of
$\mathcal{L}$.

By Theorem~\ref{thm:semialg}, both $R_{p,1}$ and $R_{p,*}$ are
semi-algebraic. The constraint indicator, quadratic terms, and
linear terms in \eqref{ch3:Ratio-ADMMP-Lag} are also
semi-algebraic. Hence $\mathcal{L}$ is semi-algebraic and therefore
satisfies the KL property. Applying the KL convergence result for
nonconvex ADMM \cite{wang2019global}, the entire primal--dual
sequence converges to a critical point $(X^*,Y^*,U^*)$ of
$\mathcal{L}$.

Finally, primal feasibility $X^*=Y^*$ together with the limiting
optimality conditions of the ADMM subproblems yields the
first-order stationarity conditions of the original powered
ratio-of-norms regularized SNMF problem.
\end{proof}

Theorem~\ref{thm:admm-convergence} concerns the idealized case in
which the inner subproblems are solved exactly. In the numerical
implementation, the inner ADMM iteration is terminated at a
prescribed tolerance $\eta$. Extending the convergence result to
such inexact inner solves requires additional control of the
accumulated inner errors and is beyond the scope of this work.

\section{Numerical Experiments}\label{sec:exp}
In this section, we present numerical experiments on both real and synthetic datasets using Algorithms~\ref{alg:RatioDCA} and~\ref{alg:Ratio-ADMMP}. We first evaluate index identification on two synthetic datasets under varying noise levels. We then apply the proposed algorithms to the classification of hand gesture images, including both binary and grayscale data.

To evaluate the proposed methods, we compare the proposed algorithms with several benchmark methods, including SPA~\cite{Gillis2014}, FGNSR~\cite{Gillis2018}, MERIT~\cite{Nguyen2022}, VCA~\cite{Nascimento2005}, and XRAY~\cite{Kumar2013}. In our algorithms, the regularization and penalty parameters are selected by Bayesian optimization, with the detailed tuning procedure described in Section~\ref{sec:bo-procedure}. The maximal inner and outer iteration numbers are fixed as $J = 10$ and $K = 100$, respectively, with both tolerances set to $10^{-5}$. For the DCA-based Algorithm~\ref{alg:RatioDCA}, we restrict the experiments to $p=1$
to avoid an extensive solver-dependent parameter study. The effect
of the power parameter is instead systematically investigated using
ADMM-based Algorithm~\ref{alg:Ratio-ADMMP} for $p=1,2,3,4$. See Section~\ref{sec:dis} for further discussion on the sensitivity of our algorithms to the power parameter $p$ and the maximum numbers of inner and outer iterations.

To assess performance in terms of identification accuracy and approximation quality, we employ the following metrics throughout this section.
\begin{enumerate}
\item Given a true index set $\cK_t^*$ and an approximated index set $\cK_t$
produced by an algorithm on trial $t$, the \emph{success rate} (SR) over
$T$ independent trials is defined as
\[
\mathrm{SR}
=
\frac{1}{T}\sum_{t=1}^{T}
\I_{\{\cK_t=\cK_t^*\}},
\]
where $\I_{\{\cK_t=\cK_t^*\}}=1$ if $\cK_t=\cK_t^*$ and $0$ otherwise.
Equality is judged as set equality, not ordered tuple equality.

\item For a recovered basis column $\vw$ and a true basis column $\vw^*$, the \emph{mean removed spectral angle} (MRSA) is
$$
\text{MRSA}(\vw, \vw^*) \;=\; \frac{100}{\pi}\,\arccos\!\left(\frac{\la \vw - \overline{\vw},\, \vw^* - \overline{\vw}^* \ra}{\|\vw - \overline{\vw}\|\,\|\vw^* - \overline{\vw}^*\|}\right),
$$
where $\overline{\vw}$ is the mean of the entries of $\vw$. The value is normalized to the range $[0,100]$, with $0$ indicating a perfect match. To obtain a single summary value per experiment, MRSA is averaged over all matched pairs of recovered and true columns after solving an optimal assignment problem between the two index sets. MRSA is used in the synthetic experiments in Section~\ref{sec:exp}.

\item Given a data matrix $M$ and an approximated index set $\cK$, the \emph{relative approximation error} (RAE) is
$$
\text{RAE} \;=\; 1 \;-\; \frac{\min_{H\geq 0}\|M - M(:,\cK)\,H\|_F}{\|M\|_F}.
$$
The inner minimization is solved using hierarchical alternating nonnegative least squares~\cite{Gillis2011}. RAE measures how well the selected columns $M(:,\mathcal K)$ reconstruct
$M$ through nonnegative combinations: values close to $1$ indicate a
near-exact reconstruction, whereas values close to $0$ correspond to
large residual errors.

\end{enumerate}

All numerical experiments are implemented in MATLAB~R2025a on a Dell desktop equipped with a 12th-generation Intel Core i5-12500T (2.00~GHz) processor and 16~GB of RAM.

\subsection{Synthetic Data Experiments}\label{sec:SynthExp}
For all synthetic experiments, the proposed algorithms are initialized with $X_0=I_n$. We consider the following two types of synthetic data~\cite{Gillis2014}.
\begin{itemize}
\item \textbf{Type 1: midpoint data.} Let $M_0=WH \in \mathbb{R}^{50 \times 55}$ with rank $r = 10$. The anchor matrix $W \in \mathbb{R}^{50 \times 10}$ is a random matrix whose entries are uniformly sampled on $[0,1]$, i.e., \texttt{rand} in MATLAB, and whose columns are normalized to have unit $\ell_1$-norm. The matrix $H = [I_r \mid H_{\text{mid}}]$, where the $\binom{r}{2} = 45$ columns of $H_{\text{mid}}$ represent pairwise midpoints of the columns of $W$. Thus, each non-anchor column lies exactly between two anchors, making identification particularly challenging. Let $\bar{\vw}=\frac{1}{r}\sum_{\ell=1}^r W(:,\ell)$ denote the centroid of the columns of $W$. A structured perturbation $N_0$ is added only to the non-anchor columns via $N_0(:,1:r)=0$ and $N_0(:,j)=M_0(:,j)-\bar{\vw}$ for $j=r+1,\ldots,n$, and is globally rescaled as $N=\varepsilon N_0/\|N_0\|_F$. We then set $M=M_0+N$, so that $\|N\|_F=\varepsilon$. Finally, the columns of $M$ are randomly permuted.

\item \textbf{Type 2: Dirichlet data.} Let $M_0=WH\in\mathbb{R}^{50\times100}$ with rank $r=10$. The anchor matrix $W$ is generated and normalized as in Type 1. The first $r$ columns of $H$ form $I_r$, and the remaining $90$ columns are sampled from the Dirichlet distribution with parameter vector $\mathbf{1}_r$. Gaussian noise is added to all columns and globally rescaled to satisfy $\|N\|_F=\varepsilon\|M_0\|_F$, yielding $M=M_0+N$, followed by a random column permutation.
\end{itemize}

\subsubsection{Noise-free Data}
We first evaluate the proposed methods in the noise-free setting. Tables~\ref{ch3-tab:NoiselessRecoveryL1}--\ref{ch3-tab:NoiselessRecoveryNuclear} report recovery accuracies in the noise-free case for the proposed algorithms after Bayesian optimization tuning of the algorithmic parameters. ADMM-P achieves exact recovery on Type~1 for $p\geq2$ and on Type~2 for most tested values of $p$, with the exception of $p=3$ underv$R_{p,1}$, where the recovery accuracy is $0.38$. DCA, evaluated at $p=1$, perfectly recovers Type~2 under both regularizers but performs poorly on the more challenging Type~1 data, attaining accuracies of $0.68$ under $R_{p,1}$ and $0$ under $R_{p,*}$. The weaker Type~1 performance of DCA suggests greater sensitivity to the challenging midpoint geometry, particularly under the nuclear-norm regularizer.

\begin{table}[!htbp]
    \centering
    \caption{Noise-free recovery success rates and runtimes for Algorithms~\ref{alg:RatioDCA}--\ref{alg:Ratio-ADMMP} with $R_{p,1}$-regularization on Type~1 and Type~2 data.}
    \label{ch3-tab:NoiselessRecoveryL1}
    \begin{tabular}{ccc|cc}
         \toprule
         \textbf{Algorithm} & \multicolumn{2}{c}{\textbf{Type 1 Data}} & \multicolumn{2}{c}{\textbf{Type 2 Data}} \\
         \cmidrule(lr){1-1} \cmidrule(lr){2-3} \cmidrule(lr){4-5}
         DCA & \textbf{Success Rate} & \textbf{Runtime} & \textbf{Success Rate} & \textbf{Runtime} \\
         \cmidrule(lr){1-1} \cmidrule(lr){2-5}
          $p=1$ & $0.68$ & $0.0633$&  $1$ & $0.0962$ \\
         \cmidrule(lr){1-1} \cmidrule(lr){2-5}
         ADMM-P & \textbf{Success Rate} & \textbf{Runtime} & \textbf{Success Rate} & \textbf{Runtime} \\
         \cmidrule(lr){1-1} \cmidrule(lr){2-5}
         $p = 1$ & $0.32$ & $0.0396$ &  $1$ & $0.0021$\\
         $p = 2$ & $1$ & $0.0868$ & $1$ & $0.518$\\
         $p = 3$ & $1$ & $0.0657$ & $0.38$ & $0.0216$\\
         $p = 4$ & $1$ & $0.0701$ & $1$ & $0.2014$ \\  \hline
    \end{tabular}
\end{table}

\begin{table}[!h]
    \vspace{-0.5\baselineskip}
    \centering
    \caption{Noise-free recovery success rates and runtimes for Algorithms \ref{alg:RatioDCA}--\ref{alg:Ratio-ADMMP} with $R_{p,*}$-regularization on Type~1 and Type~2 data.}
    \label{ch3-tab:NoiselessRecoveryNuclear}
    \begin{tabular}{ccc|cc}
         \toprule
         \textbf{Algorithm} & \multicolumn{2}{c}{\textbf{Type 1 Data}} & \multicolumn{2}{c}{\textbf{Type 2 Data}} \\
         \cmidrule(lr){1-1} \cmidrule(lr){2-3} \cmidrule(lr){4-5}
         DCA & \textbf{Success Rate} & \textbf{Runtime} & \textbf{Success Rate} & \textbf{Runtime} \\
         \cmidrule(lr){1-1} \cmidrule(lr){2-5}
          $p=1$ & 0 & $0.0346$&  $1$ & $0.1004$ \\
         \cmidrule(lr){1-1} \cmidrule(lr){2-5}
         ADMM-P & \textbf{Success Rate} & \textbf{Runtime} & \textbf{Success Rate} & \textbf{Runtime} \\
         \cmidrule(lr){1-1} \cmidrule(lr){2-5}
         $p = 1$ & $0.60$ & $0.6788$ &  $1$ & $2.1212$\\
         $p = 2$ & $1$ & $0.5421$ & $1$ & $2.5972$\\
         $p = 3$ & $1$ & $0.9027$ & $1$ & $2.6990$\\
         $p = 4$ & $1$ & $0.9391$ & $1$ & $1.9866$ \\  \hline
    \end{tabular}
\end{table}

\subsubsection{Noisy Data}
To analyze the robustness of Algorithm~\ref{alg:Ratio-ADMMP}, we consider noisy data by adding noise to the synthetic datasets and compare the performance. We let our noise level $\ep$ be sampled from a logarithmically spaced grid over $[10^{-2}, 1]$ with 20 points, and at each noise level, we run 50 independent trials and report the average results. For Type~1 data, parameters selected for the low-to-moderate noise regime do not remain effective at higher noise levels. We therefore use two parameter configurations in the reported experiments: the
original configuration for $\ep<0.1833$ and a separately selected high-noise configuration for $\ep\geq0.1833$. The latter was selected from additional evaluations of BO-generated candidates, as described in Section~\ref{sec:bo-procedure}. For Type~2 data, a
single parameter configuration is used throughout the noise sweep.

For each $p \in \{1, 2, 3, 4\}$, the configurations producing the highest average success rates and corresponding runtimes for Type~1 and Type~2 data are reported in Figures~\ref{ch3-fig:SynthExp-Midpoint-selectedAlgs} and~\ref{ch3-fig:SynthExp-dir_rand-selectedAlgs}, respectively.

\begin{figure}[htbp!]
\centering
\begin{tabular}{cc}
\includegraphics[width=0.48\textwidth]{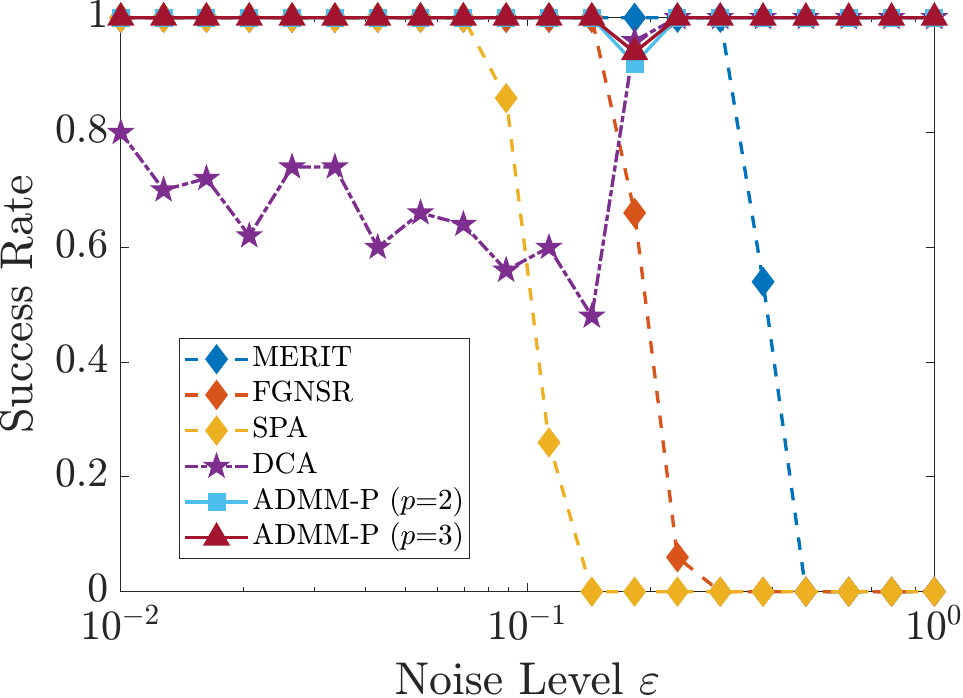}
&\includegraphics[width=0.48\textwidth]{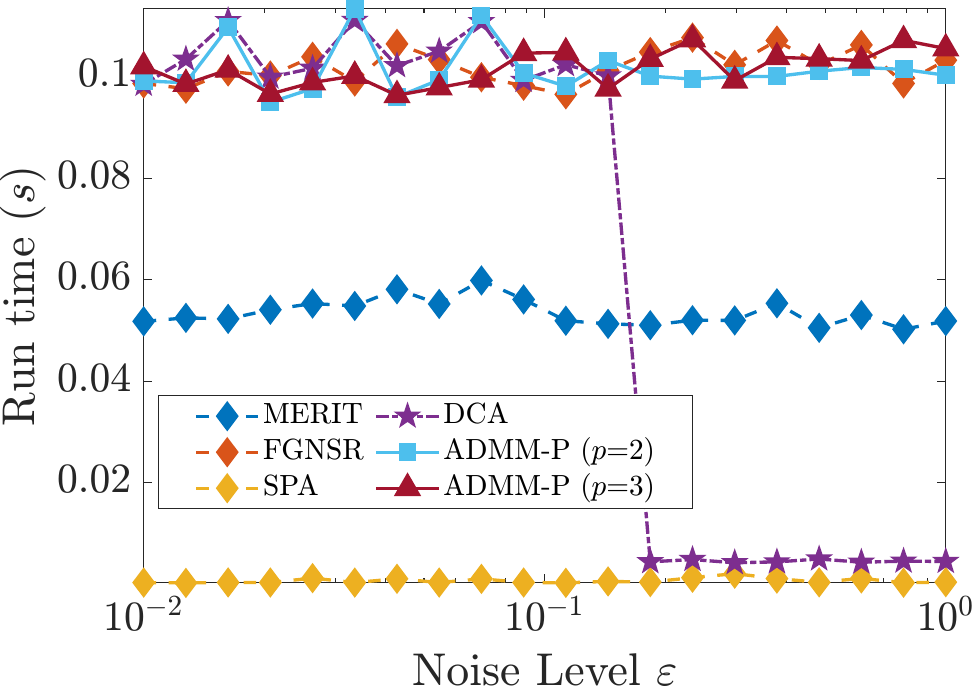}
\\
{(a) $R_{p,1}$ success rate }&{(b) $R_{p,1}$ runtime}\\[10pt]
\includegraphics[width=0.48\textwidth]{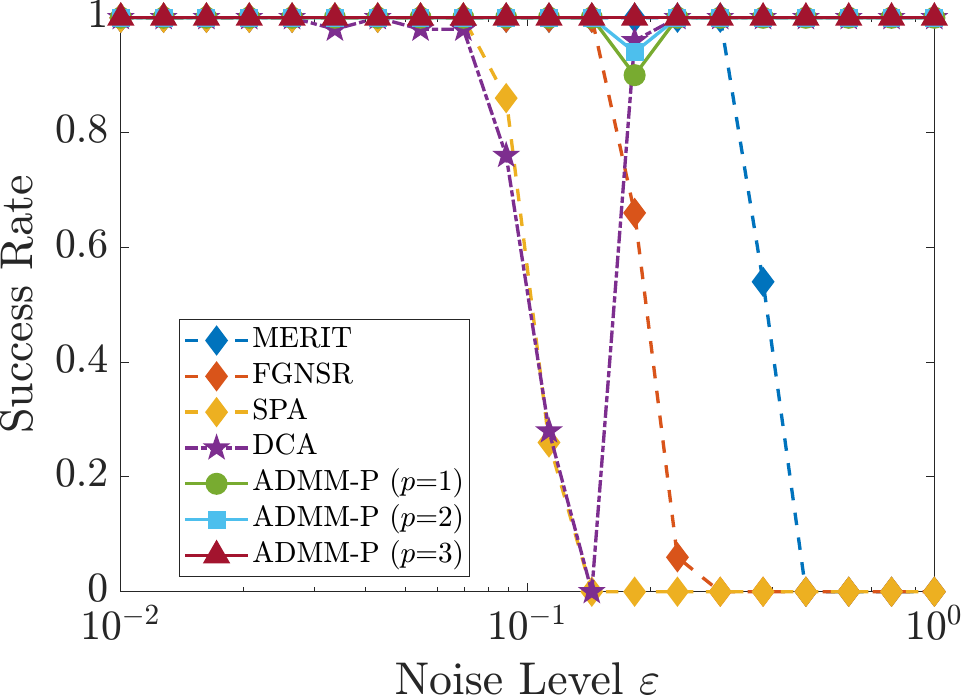}
&\includegraphics[width=0.48\textwidth]{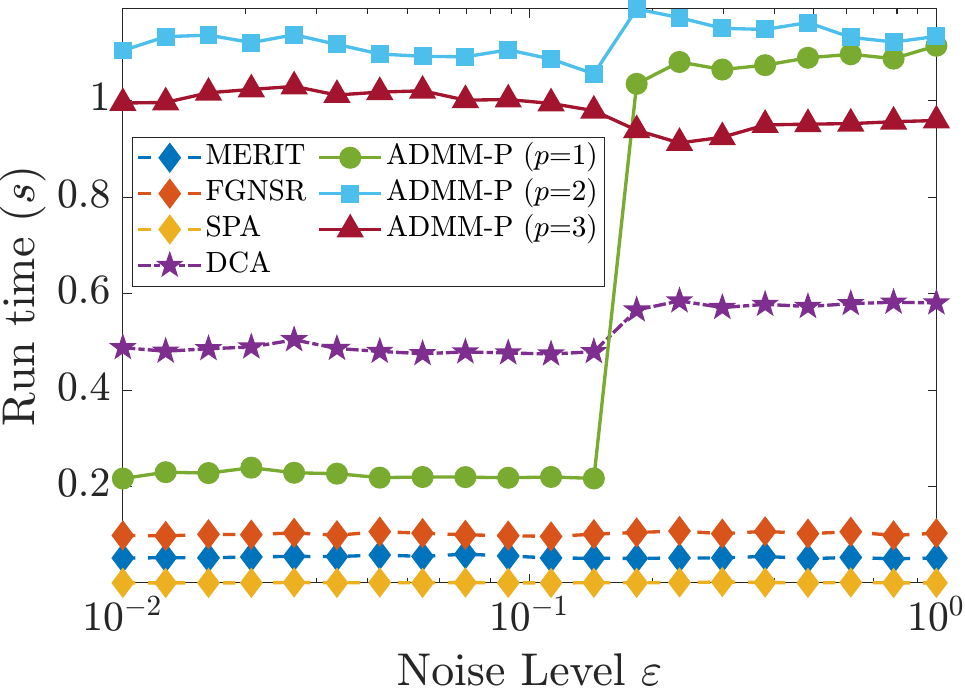}\\
(c) $R_{p,*}$ success rate&(d) $R_{p,*}$ runtime
\end{tabular}
\caption[Performance comparison on Type 1 data]{Performance comparison of the proposed algorithms and baseline methods on Type 1 data. }
\label{ch3-fig:SynthExp-Midpoint-selectedAlgs}
\end{figure}

\begin{figure}[htbp!]
\begin{tabular}{cc}
\includegraphics[width=.485\textwidth]{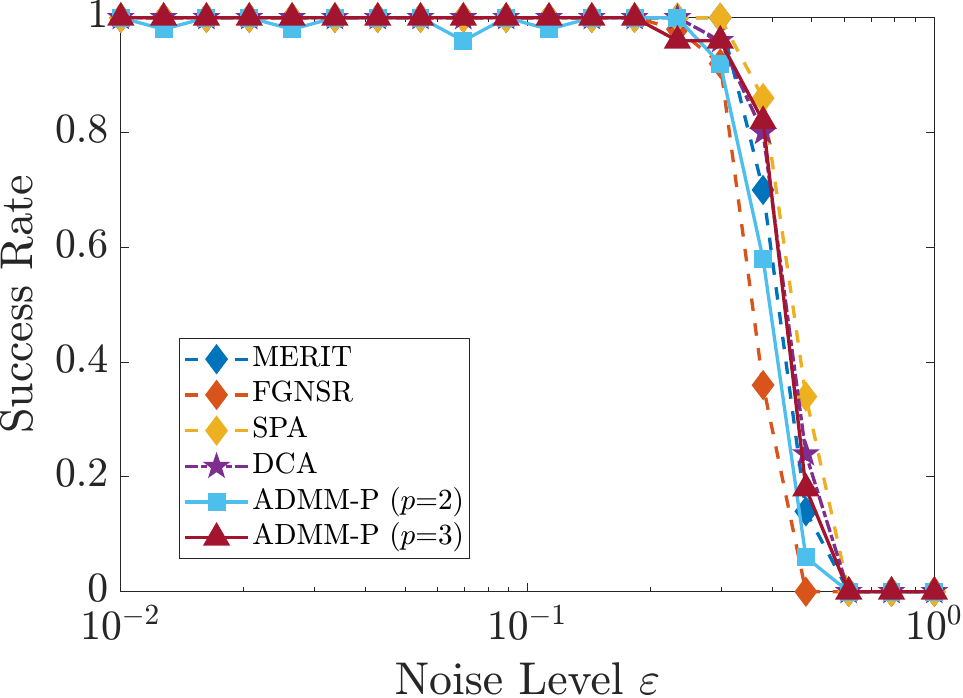}
&   \includegraphics[width=0.485\linewidth]{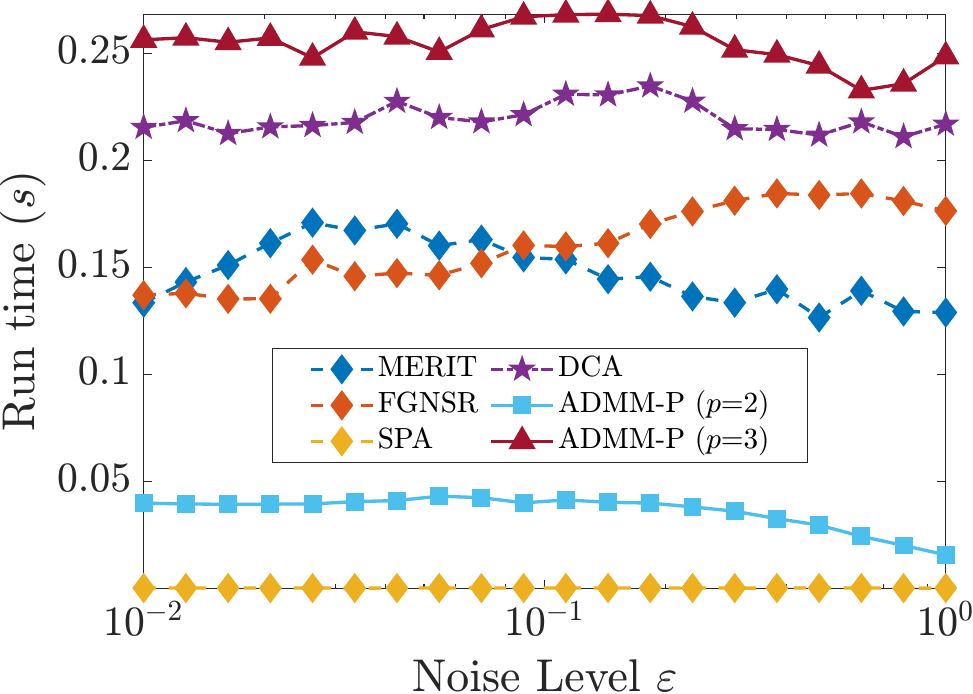}\\
    {(a) $R_{p,1}$ success rate}&
    {(b) $R_{p,1}$ runtime}\\[8pt]
\includegraphics[width=0.485\linewidth]{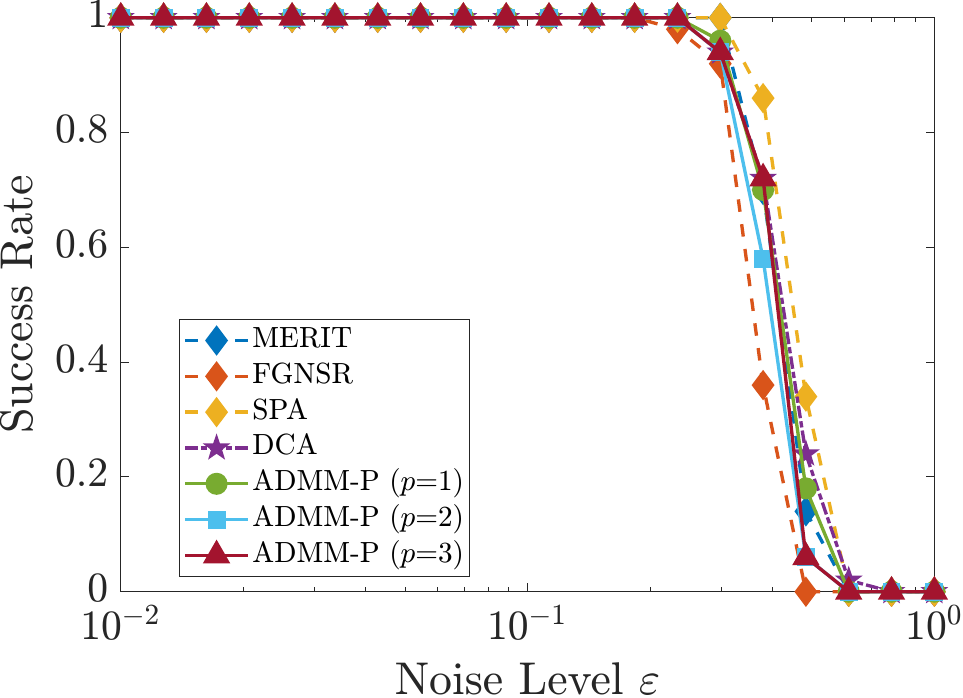}
&
\includegraphics[width=0.485\linewidth]{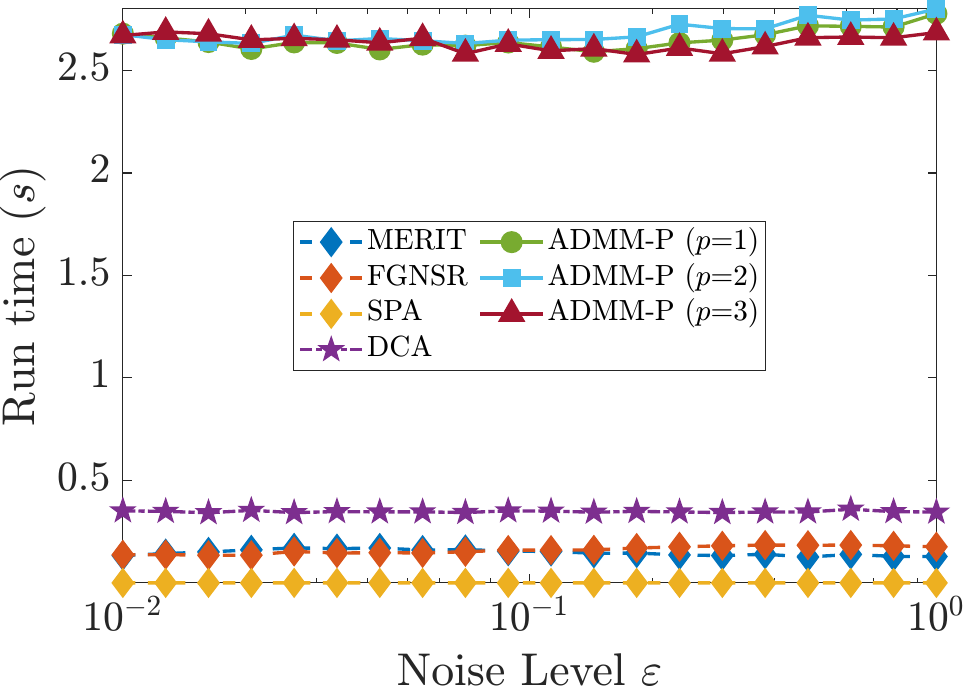}\\
    {(c) $R_{p,*}$ success rate}&
    {(d) $R_{p,*}$ runtime}
\end{tabular}
\caption{Performance comparison of the proposed algorithms and baseline methods on Type 2 data.}
\label{ch3-fig:SynthExp-dir_rand-selectedAlgs}
\end{figure}

On Type~1 data with the $R_{p,1}$ regularizer (Figure~\ref{ch3-fig:SynthExp-Midpoint-selectedAlgs}), ADMM-P maintains a perfect success rate across the entire noise sweep $\varepsilon \in [10^{-2}, 1]$. The baselines SPA and FGNSR fail near $\varepsilon \approx 10^{-1.5}$, MERIT degrades near $\varepsilon \approx 0.5$, and DCA holds an intermediate plateau around $0.6$--$0.8$ at low noise before dropping near $\varepsilon \approx 10^{-1}$. Runtime is approximately constant in the noise level, with ADMM-P at $\approx 0.1$s, DCA at $\approx 0.05$s, and the baselines essentially free. Under the $R_{p,*}$ regularizer (Figure~\ref{ch3-fig:SynthExp-Midpoint-selectedAlgs}) the qualitative picture is similar for ADMM-P, which maintains a nearly perfect success rate throughout. However DCA mirrors the success rate of SPA until $\ep \approx 10^{-1}$, after which the parameters change to the high noise parameters causing DCA to have a perfect success rate for the remainder of the sweep. This change in parameter for DCA and ADMM-P ($p=1$) can be seen in the runtime, the runtime cost rises by an order of magnitude (ADMM-P $\approx 0.6$--$1.0$s, DCA $\approx 0.5$s) owing to the SVD required by the nuclear-norm subproblem.

On Type~2 data (Figure~\ref{ch3-fig:SynthExp-dir_rand-selectedAlgs}), every method maintains a success rate near unity until $\varepsilon \approx 0.3$, beyond which all curves collapse together. SPA has the fastest runtime, and in the case of the $R_{p,1}$ regularization Algorithm~\ref{alg:Ratio-ADMMP} has a lower runtime than all other algorithms with the exception of SPA. Runtime under $R_{p,1}$ remains under $0.4$s for all methods, while the $R_{p,*}$ variants again pay an order-of-magnitude penalty (up to $\approx 2.7$s for ADMM-P at the high-noise end).

The synthetic experiments support the conclusion that the ratio-of-norms formulation provides substantially improved noise robustness on midpoint geometries while remaining competitive elsewhere. Overall, the entrywise regularizer provides accuracy comparable to
the nuclear-norm regularizer in these experiments while requiring
substantially less computation.

\subsection{Hand Gesture Classification}\label{sec:HGexp}
In this experiment, we apply our algorithms to dimensionality reduction
for hand gesture image classification. We test two datasets: binary
images downloaded from~\cite{Goyal2020} and grayscale images
from~\cite{Hoang2020}. We compare Algorithms~\ref{alg:RatioDCA}--\ref{alg:Ratio-ADMMP} against three SNMF algorithms: SPA~\cite{Araujo2001},
MERIT~\cite{Nguyen2022}, and FGNSR~\cite{Gillis2018}. For the proposed
algorithms, we use a common SPA+NNLS warm start, where SPA first selects
the anchor columns and NNLS is then used to construct the initial matrix
$X_0$. SPA-based initialization is also used for MERIT and FGNSR
according to their implementations.

\subsubsection{Binary Images}
For the binary image data downloaded from \cite{Goyal2020}, there are three hand gestures in the database: fist, open-hand, and two fingers which have 2003, 2010, and 2005 images, respectively. Each binary image is of size $150 \ti 150$; see Figure~\ref{fig:bin hg all} for sampled images.
We resize each image to $75\ti 75$ and then use PCA to extract the 50 most representative images from each hand gesture. We then compute the HOG on the reduced data set of $150$ images using the MATLAB command \texttt{extractHOGFeatures} and the LBP via \texttt{extractLBPFeatures} and then obtain matrices of size $150 \ti 1292$ and $150 \ti 2048$, respectively.

Next, we test all the comparing methods with varying factorization ranks, i.e., $r$ being  10\%, 20\%, 30\%, 40\%, and $50\%$ of the columns of a data matrix. For the post-processing step, we select the $r$ largest diagonal entries of $X$. The parameters for Algorithms \ref{alg:RatioDCA}--\ref{alg:Ratio-ADMMP} were tuned via Bayesian optimization and are given in Table~\ref{ch3-tab:selected-params-bin}.

After the column index set $\cK$ is obtained, we perform classification  on the reduced feature matrix with a specified factorization rank. Given the $m\ti r$ reduced image-feature matrix, we split the data, holding $80\%$ of the images as training data, and the other $20\%$ of the images as testing data using MATLAB command \texttt{cvpartition}. We then train a SVM model for classification of the binary images on the training reduced image-feature matrix in MATLAB using \texttt{fitcecoc}. We run $50$ trials for each $r$ value, and average the accuracies given from each trial. We report the accuracies of the highest performing variants of Algorithms \ref{alg:RatioDCA}--\ref{alg:Ratio-ADMMP} in Table~\ref{ch3-tab:binary-acc-best} with the corresponding runtimes given in Table~\ref{ch3-tab:binary-time-best}.

The proposed methods achieve classification accuracies comparable to or higher than the baseline methods, particularly at larger factorization ranks. As shown in Table~\ref{ch3-tab:binary-time-best}, their runtimes also vary less with the factorization rank than that of SPA for the LBP features, although this empirical observation alone does not imply rank-independent complexity.

\begin{table}[htbp]
\centering
\caption{Best parameters in Algorithms~\ref{alg:RatioDCA}-\ref{alg:Ratio-ADMMP} with $R_{p,1}$ regularizer for binary data.}\label{ch3-tab:selected-params-bin}
\begin{adjustbox}{max width=\textwidth}
\begin{tabular}{c|ccccc}
\hline
\multicolumn{6}{c}{\textbf{HOG Features}} \\
\hline
 & r1 & r2 & r3 & r4 & r5 \\
\hline
\multicolumn{6}{c}{\textbf{DCA}} \\
$\lambda$ & $1.63\times 10^{3}$ & $1.16\times 10^{-5}$ & $1.33\times 10^{-3}$ & $1.16\times 10^{-5}$ & $5.97\times 10^{-3}$ \\
$\rho$    & $5.5571$ & $1.78\times 10^{-5}$ & $2.97\times 10^{-5}$ & $1.78\times 10^{-5}$ & $1.00\times 10^{-5}$ \\
$\beta$   & $2.85\times 10^{3}$ & $1.43\times 10^{-4}$ & $1.50\times 10^{-4}$ & $1.43\times 10^{-4}$ & $1.07\times 10^{-5}$ \\
\hline
\multicolumn{6}{c}{\textbf{ADMM-P ($p=3$)}} \\
$\lambda$ & $1.90\times 10^{-3}$ & $1.62\times 10^{-5}$ & $1.09\times 10^{-4}$ & $4.31\times 10^{-2}$ & $7.80\times 10^{-4}$ \\
$\rho_1$  & $0.4283$ & $6.1829$ & $2.1610$ & $15.5746$ & $1.0356$ \\
$\rho_2$  & $1.00\times 10^{-5}$ & $1.39\times 10^{-2}$ & $2.23\times 10^{-3}$ & $1.89\times 10^{-2}$ & $7.57\times 10^{-3}$ \\
$\rho_3$  & $3.6105$ & $51.4323$ & $6.93\times 10^{-2}$ & $0.2383$ & $4.4229$ \\
\hline
\noalign{\vskip 6pt}
\multicolumn{6}{c}{\textbf{LBP Features}} \\
\hline
 & r1 & r2 & r3 & r4 & r5 \\
\hline
\multicolumn{6}{c}{\textbf{DCA}} \\
$\lambda$ & $2.0918$ & $2.10\times 10^{3}$ & $0.1338$ & $1.16\times 10^{-5}$ & $2.63\times 10^{3}$ \\
$\rho$    & $8.39\times 10^{-3}$ & $0.3921$ & $0.2989$ & $1.78\times 10^{-5}$ & $0.7461$ \\
$\beta$   & $40.8187$ & $5.61\times 10^{-4}$ & $1.13\times 10^{-5}$ & $1.43\times 10^{-4}$ & $0.2251$ \\
\hline
\multicolumn{6}{c}{\textbf{ADMM-P ($p=1$)}} \\
$\lambda$ & $1.7143$ & $3.35\times 10^{-2}$ & $1.7143$ & $1.7143$ & $1.7143$ \\
$\rho_1$  & $251.0359$ & $9.26\times 10^{-4}$ & $251.0359$ & $251.0359$ & $251.0359$ \\
$\rho_2$  & $177.9946$ & $1.04\times 10^{-5}$ & $177.9946$ & $177.9946$ & $177.9946$ \\
$\rho_3$  & $514.1343$ & $1.77\times 10^{-2}$ & $514.1343$ & $514.1343$ & $514.1343$ \\
\hline
\end{tabular}
\end{adjustbox}
\end{table}

\begin{figure}[htbp]
     \centering
     \setlength{\tabcolsep}{1pt}
\begin{tabular}{ccc}
         \includegraphics[width=0.2\textwidth]{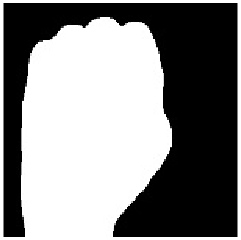}
&
         \includegraphics[width=0.2\textwidth]{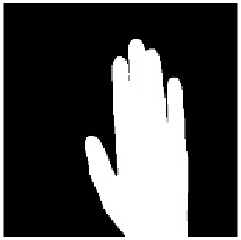}
&
         \includegraphics[width=0.2\textwidth]{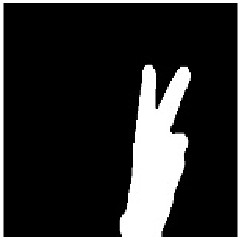}
\\
         (a) Fist & (b) Open-hand & (c) Two fingers
     \end{tabular}
        \caption{Sample images in the binary hand gesture data.}
        \label{fig:bin hg all}
\end{figure}

\begin{table}[!htbp]
\centering
\caption{Classification accuracy comparison on binary hand gesture data using HOG and LBP features.
}
\label{ch3-tab:binary-acc-best}
\begin{tabular}{clccccc}
\toprule
& \textbf{Features} & \multicolumn{5}{c}{\textbf{Factorization Rank $r$}} \\
\cmidrule(lr){3-7}
& & \textbf{130} & \textbf{259} & \textbf{388} & \textbf{517} & \textbf{646} \\
\multicolumn{7}{l}{\textbf{HOG}} \\
\midrule
& SPA             & 0.6333 & 0.8000 & 0.8333 & 0.9000 & 0.9667 \\
& FGNSR           & 0.6667 & 0.8333 & 0.8333 & 0.9333 & 0.9667 \\
& MERIT           & 0.6667 & 0.6667 & 0.8667 & 0.9000 & 0.9667 \\
\cmidrule(lr){2-7}
& DCA     & 0.9867          & \textbf{0.9960} & \textbf{0.9960} & \textbf{0.9947}          & \textbf{0.9960} \\
& ADMM-P  & \textbf{0.9960} & \textbf{0.9960} & 0.9947 & \textbf{0.9947} & 0.9947 \\
\cmidrule(lr){2-7}
& Random          & 0.7333 & 0.8333 & 0.9667 & 0.9667 & 1.0000 \\
& All Features    & 1.0000 & 1.0000 & 1.0000 & 1.0000 & 1.0000 \\
\cmidrule(lr){3-7}
& & \textbf{205} & \textbf{410} & \textbf{615} & \textbf{820} & \textbf{1024} \\
\multicolumn{7}{l}{\textbf{LBP}} \\
\midrule
& SPA             & 0.9000 & 0.9333 & \textbf{0.9667} & \textbf{0.9667} & \textbf{0.9667} \\
& FGNSR           & 0.9000 & 0.9333 & \textbf{0.9667} & \textbf{0.9667} & \textbf{0.9667} \\
& MERIT           & 0.8333 & 0.9000 & \textbf{0.9667} & \textbf{0.9667} & \textbf{0.9667} \\
\cmidrule(lr){2-7}
& DCA     & \textbf{0.9613} & 0.9640          & 0.9653 & 0.9333 & 0.9653 \\
& ADMM-P  & \textbf{0.9613} & \textbf{0.9653} & 0.9653 & 0.9653 & 0.9653 \\
\cmidrule(lr){2-7}
& Random          & 0.5667 & 0.7000 & 0.6667 & 0.8333 & 0.8000 \\
& All Features    & 0.9000 & 0.9000 & 0.9000 & 0.9000 & 0.9000 \\
\bottomrule
\end{tabular}
\end{table}

\begin{table}[!htbp]
\centering
\caption{Runtime (seconds) comparison of Algorithms~\ref{alg:RatioDCA} and \ref{alg:Ratio-ADMMP} versus baseline methods on binary hand gesture data.
}
\label{ch3-tab:binary-time-best}
\begin{tabular}{clccccc}
\toprule
& \textbf{Features} & \multicolumn{5}{c}{\textbf{Factorization Rank $r$}} \\
\cmidrule(lr){3-7}
& & \textbf{130} & \textbf{259} & \textbf{388} & \textbf{517} & \textbf{646} \\
\multicolumn{7}{l}{\textbf{HOG}} \\
\midrule
& SPA & 0.0387 & 0.0583 & 0.0822 & 0.1078 & 0.1979 \\
& FGNSR & 30.12 & 27.57 & 21.95 & 18.36 & 15.28 \\
& MERIT & 38.08 & 35.26 & 27.27 & 18.73 & 15.64 \\
\cmidrule(lr){2-7}
& DCA & 30.34 & 18.18 & 21.90 & 18.52 & 35.18 \\
& ADMM-P  & 26.21 & 26.83 & 27.56 & 27.09 & 26.94 \\
\cmidrule(lr){3-7}
& & \textbf{205} & \textbf{410} & \textbf{615} & \textbf{820} & \textbf{1024} \\
\multicolumn{7}{l}{\textbf{LBP}} \\
\midrule
& SPA & 0.0962 & 0.2260 & 0.2670 & 0.5365 & 0.6009 \\
& FGNSR & 14.51 & 11.26 & 10.41 & 7.79 & 5.73 \\
& MERIT & 127.97 & 103.16 & 85.88 & 58.98 & 43.90 \\
\cmidrule(lr){2-7}
& DCA & 29.29 & 97.18 & 97.63 & 61.57 & 98.78 \\
& ADMM-P & 72.09 & 66.71 & 71.01 & 73.95 & 71.06 \\
\bottomrule
\end{tabular}
\end{table}

\subsubsection{Grayscale Images}
For this experiment, we use the HGM-4 multi-cameras dataset \footnote{Available at \url{https://data.mendeley.com/datasets/jzy8zngkbg/4}}. In particular, we choose five classes of images representing hand gestures for the letters A, B, C, H, and Y, as these gestures are most distinguishable for fair comparisons; see one example image for each gesture in Fig.~\ref{fig:gs_hg}. Each gesture class has 40 images, each image with size $160 \ti 90$. We reduce the image size to $80\ti 45$, then we  expand the data set by image rotation. Specifically, for each image, we rotate each image by angles in $\{-2^\circ,-1^\circ,1^\circ,2^\circ\}$ and then find the $25$ best representative images using PCA for each class. Then we extract the HOG/LBP features as we did with the binary data, and generate matrices of size $125 \ti 1296$ and $125 \ti 1696$, respectively. For the post-processing step that selects the index set $\cK$, we select the row indices corresponding to the $r$ largest $\ell_2$ norms of the rows of $X$. We compute the classification accuracy for each algorithm with 50 trials. The comparison of average accuracy by using the SVM classifier for each method is illustrated in Table~\ref{ch3-tab:acc-best-svm-grayscale} and the runtime for each algorithm is given in Table \ref{ch3-tab:time-best-svm-grayscale}, which shows our method can achieve an optimal balance between classification accuracy and runtime.

On grayscale HOG features, the proposed methods again dominate at every factorization rank: ADMM-P attains accuracies between $0.87$ and $0.90$, compared with $0.32$--$0.80$ for SPA, FGNSR, and MERIT, and they exceed the full-feature SVM baseline of $0.84$ for $r \geq 260$. DCA is competitive at small $r$ but trails ADMM-P at the larger ranks. On grayscale LBP features, ADMM-P continues to outperform the baselines, with accuracies between $0.84$ and $0.86$ across $r$, while DCA is less stable, dropping to $0.56$ at $r = 679$ and $0.60$ at $r = 848$, which we attribute to the rank-specific parameter tuning for DCA differing by several orders of magnitude across rows (Table~\ref{c3-tab:selected-params-gs}). Runtimes for the proposed methods range between approximately $26$ and $70$ seconds across all configurations, comparable to FGNSR and MERIT.

Taken together, the hand gesture classification experiments show that ratio-of-norms feature selection delivers substantial accuracy gains over the baselines on HOG features on both datasets, and offers superior performance for lower ranks on LBP features. ADMM-P is the most uniformly reliable variant across configurations, and DCA exhibits the largest rank-to-rank variability. The principal cost of the proposed methods is runtime, which remains
comparable to or lower than that of FGNSR and MERIT, with ADMM-P
showing particularly stable runtimes across the factorization ranks.

\begin{figure}[]
\centering
\setlength{\tabcolsep}{2pt}
\begin{tabular}{ccccc}
\includegraphics[width=0.19\textwidth,height=.15\textheight]{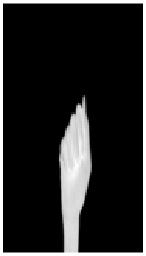}&
\includegraphics[width=0.19\textwidth,height=.15\textheight]{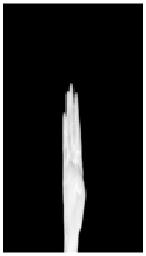}&
\includegraphics[width=0.19\textwidth,height=.15\textheight]{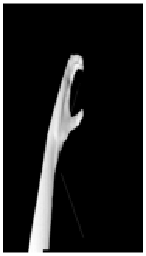}&
\includegraphics[width=0.19\textwidth,height=.15\textheight]{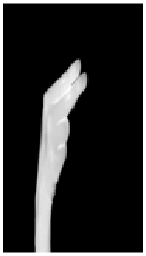}&
\includegraphics[width=0.19\textwidth,height=.15\textheight]{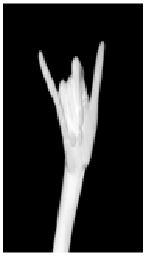}\\
(a) A & (b) B & (c) C & (d) H & (e) Y
\end{tabular}
\vspace{-6pt}
\caption{Gray-Scale hand gestures}
\label{fig:gs_hg}
\vspace{-6pt}
\end{figure}

\begin{table}[htbp]
\centering
\caption{Best parameters for selected $p$ values: Grayscale data, $R_{p,1}(X)$ regularizer}
\label{c3-tab:selected-params-gs}
\begin{adjustbox}{max width=\textwidth}
\begin{tabular}{cccccccc}
\hline
\multicolumn{8}{c}{\textbf{HOG Features}} \\
\hline
\multirow{2}{*}{\textbf{$r$}}
& \multicolumn{3}{c}{\textbf{DCA}}
& \multicolumn{4}{c}{\textbf{ADMM-P ($p=3$)}} \\
\cmidrule(lr){2-4} \cmidrule(lr){5-8}
 & \textbf{$\lambda$} & \textbf{$\rho$} & \textbf{$\beta$}
 & \textbf{$\lambda$} & \textbf{$\rho_1$} & \textbf{$\rho_2$} & \textbf{$\rho_3$} \\
\hline
r1 & $5.97\times 10^{-3}$ & $1.00\times 10^{-5}$ & $1.07\times 10^{-5}$
   & $5.16\times 10^{-4}$ & $5.34\times 10^{-2}$ & $374.2661$ & $1.42\times 10^{-5}$ \\
r2 & $5.6165$ & $17.2818$ & $202.9424$
   & $2.84\times 10^{-3}$ & $1.70\times 10^{-4}$ & $53.2401$ & $293.4317$ \\
r3 & $1.42\times 10^{-5}$ & $2.56\times 10^{-3}$ & $8.31\times 10^{-4}$
   & $0.1255$ & $2.00\times 10^{-2}$ & $154.4073$ & $1.36\times 10^{3}$ \\
r4 & $2.61\times 10^{-4}$ & $0.3510$ & $1.05\times 10^{-5}$
   & $3.13\times 10^{-2}$ & $926.9362$ & $90.6170$ & $1.71\times 10^{3}$ \\
r5 & $1.90\times 10^{-5}$ & $5.67\times 10^{-2}$ & $1.37\times 10^{-5}$
   & $2.55\times 10^{-5}$ & $4.38\times 10^{-2}$ & $5.2875$ & $29.6875$ \\
\hline
\noalign{\vskip 6pt}
\multicolumn{8}{c}{\textbf{LBP Features}} \\
\hline
\multirow{2}{*}{\textbf{$r$}}
& \multicolumn{3}{c}{\textbf{DCA}}
& \multicolumn{4}{c}{\textbf{ADMM-P ($p=4$)}} \\
\cmidrule(lr){2-4} \cmidrule(lr){5-8}
 & \textbf{$\lambda$} & \textbf{$\rho$} & \textbf{$\beta$}
 & \textbf{$\lambda$} & \textbf{$\rho_1$} & \textbf{$\rho_2$} & \textbf{$\rho_3$} \\
\hline
r1 & $3.27\times 10^{-2}$ & $1.00\times 10^{-5}$ & $3.3317$
   & $157.8754$ & $842.9703$ & $1.05\times 10^{-5}$ & $16.5790$ \\
r2 & $5.82\times 10^{-3}$ & $7.65\times 10^{-3}$ & $1.09\times 10^{-5}$
   & $4.9470$ & $2.2307$ & $1.15\times 10^{-5}$ & $12.0943$ \\
r3 & $1.1600$ & $8.16\times 10^{-3}$ & $7.79\times 10^{-5}$
   & $3.7989$ & $54.7764$ & $3.70\times 10^{-5}$ & $6.7188$ \\
r4 & $1.01\times 10^{-5}$ & $1.0346$ & $2.42\times 10^{-5}$
   & $406.4583$ & $131.4395$ & $3.2213$ & $2.29\times 10^{3}$ \\
r5 & $1.65\times 10^{-5}$ & $6.92\times 10^{-2}$ & $6.03\times 10^{-3}$
   & $406.4583$ & $131.4395$ & $3.2213$ & $2.29\times 10^{3}$ \\
\hline
\end{tabular}
\end{adjustbox}
\end{table}

\begin{table}[htbp]
\centering
\caption{Classification accuracies for grayscale dataset, $R_{p,1}(X)$ regularizer}
\label{ch3-tab:acc-best-svm-grayscale}
\begin{tabular}{clccccc}
\toprule
& \textbf{Features} & \multicolumn{5}{c}{\textbf{Factorization Rank $r$}} \\
\cmidrule(lr){3-7}
& & \textbf{130} & \textbf{260} & \textbf{389} & \textbf{519} & \textbf{648} \\
\multicolumn{7}{l}{\textbf{HOG}} \\
\midrule
& SPA             & 0.4000 & 0.4400 & 0.4400 & 0.5600 & 0.8000 \\
& FGNSR           & 0.3200 & 0.4800 & 0.7600 & 0.8000 & 0.8000 \\
& MERIT           & 0.5600 & 0.4000 & 0.6400 & 0.8000 & 0.7200 \\
\cmidrule(lr){2-7}
& DCA     & 0.7888          & {0.8880} & 0.8304          & 0.7616          & 0.7792 \\
& ADMM-P  & \textbf{0.8736} & \textbf{0.8896} & \textbf{0.8944} & \textbf{0.8896} & \textbf{0.8960} \\
\cmidrule(lr){2-7}
& Random          & 0.2400 & 0.2800 & 0.4400 & 0.6800 & 0.8000 \\
& All Features    & 0.8400 & 0.8400 & 0.8400 & 0.8400 & 0.8400 \\
\cmidrule(lr){3-7}
& & \textbf{170} & \textbf{340} & \textbf{509} & \textbf{679} & \textbf{848} \\
\multicolumn{7}{l}{\textbf{LBP}} \\
\midrule
& SPA             & 0.5200 & 0.4400 & 0.4800 & 0.4400 & 0.7200 \\
& FGNSR           & 0.5600 & 0.4400 & 0.4800 & 0.4800 & 0.7200 \\
& MERIT           & 0.2800 & 0.6000 & 0.6400 & 0.6800 & 0.7200 \\
\cmidrule(lr){2-7}
& DCA     & {0.8416} & 0.6224          & \textbf{0.8592} & 0.5552          & 0.5968 \\
& ADMM-P  & \textbf{0.8432} & \textbf{0.8560} & \textbf{0.8592} & \textbf{0.8528} & \textbf{0.8592} \\
\cmidrule(lr){2-7}
& Random          & 0.2800 & 0.4400 & 0.5200 & 0.5600 & 0.5600 \\
& All Features    & 0.8000 & 0.8000 & 0.8000 & 0.8000 & 0.8000 \\
\bottomrule
\end{tabular}
\end{table}

\begin{table}[htbp]
\centering
\caption{Runtime (seconds) for grayscale dataset--$R_{p,1}(X)$ regularizer}
\label{ch3-tab:time-best-svm-grayscale}
\begin{tabular}{clccccc}
\toprule
& \textbf{Algorithm} & \multicolumn{5}{c}{\textbf{Factorization Rank $r$}} \\
\cmidrule(lr){3-7}
& & \textbf{130} & \textbf{260} & \textbf{389} & \textbf{519} & \textbf{648} \\
\multicolumn{7}{l}{\textbf{Gray-Scale HOG}} \\
\midrule
& SPA & 0.0946 & 0.1212 & 0.1487 & 0.1770 & 0.1995 \\
& FGNSR & 33.03 & 29.99 & 21.38 & 16.93 & 14.24 \\
& MERIT & 37.63 & 35.48 & 27.88 & 16.93 & 13.74 \\
\cmidrule(lr){2-7}
& DCA & 41.32 & 34.28 & 35.35 & 34.48 & 35.54 \\
& ADMM-P & 28.39 & 26.76 & 28.32 & 27.16 & 25.92 \\
\cmidrule(lr){3-7}
& & \textbf{170} & \textbf{340} & \textbf{509} & \textbf{679} & \textbf{848} \\
\multicolumn{7}{l}{\textbf{Gray-Scale LBP}} \\
\midrule
& SPA & 0.0753 & 0.2696 & 0.1391 & 0.1789 & 0.4812 \\
& FGNSR & 65.01 & 49.69 & 40.78 & 35.67 & 30.08 \\
& MERIT & 56.30 & 53.52 & 36.33 & 27.77 & 21.40 \\
\cmidrule(lr){2-7}
& DCA & 12.63 & 62.49 & 52.01 & 64.58 & 70.12 \\
& ADMM-P & 45.98 & 44.63 & 45.10 & 3.34 & 3.41 \\
\bottomrule
\end{tabular}
\end{table}

\section{Discussion}\label{sec:dis}
In this section, we present parameter selection in our experiments via Bayesian optimization, examine the behavior of Algorithm~\ref{alg:RatioDCA} and Algorithm~\ref{alg:Ratio-ADMMP} through a sequence of
ablation studies, analyze their per-iteration computational cost, and discuss a limitation of the high-noise synthetic results. The studies isolate the four main design parameters of the ratio-of-norms framework, including the power $p$ in the numerator norm, the inner and outer iteration budgets $J$ and $K$, the matrix dimensions, and the regularizer choice ($R_{p,1}$ versus $R_{p,*}$), and then characterize how each interacts with the choice of optimization strategy. Throughout, we use projection onto $\Om$ as the representative constraint scheme and average results over the trial counts indicated in each subsection.

\subsection{Parameter Selection via Bayesian Optimization}
\label{sec:bo-procedure}

We use Bayesian optimization (BO), a sequential model-based approach for
optimizing expensive black-box objectives~\cite{snoek2012practical,shahriari2015taking},
to select the parameters of the proposed algorithms. Specifically, we use
MATLAB's \texttt{bayesopt} routine with the \texttt{expected-improvement-plus}
acquisition function. All tuned parameters are searched over
$[10^{-5},3\times10^3]$ on a logarithmic scale. For each tuning
condition $c$, the BO objective is
\begin{equation}\label{ch3:bo-objective}
    \widehat{\mathcal L}_c(\theta)
    =\frac{1}{T}\sum_{t=1}^T \ell(\theta;\xi_t),
\end{equation}
where $\theta$ denotes the parameters to be tuned, $\ell$ is the
experiment-specific loss, and $\xi_t$ represents the randomness in the
$t$-th trial.

\subsubsection{Synthetic experiments.}
For the experiments in Section~\ref{sec:SynthExp}, we tune
$\theta=(\lambda,\rho,\beta)$ for DCA and
$\theta=(\lambda,\rho_1,\rho_2,\rho_3)$ for ADMM-P.
All parameters are searched over $[10^{-5},3\times10^3]$ on a
logarithmic scale. The experimental conditions include the data type
(Type~1 or Type~2), noise level $\ep$, and regularizer
($R_{p,1}$ or $R_{p,*}$); for ADMM-P, we consider $p\in\{1,2,3,4\}$, while DCA is evaluated with $p=1$.

Bayesian optimization is first performed separately at 20 logarithmically
spaced noise levels over $[10^{-2},1]$. At each noise level, the BO
objective is the mean exact-recovery error over independently generated
synthetic instances,
\begin{equation}\label{ch3:ch3-loss-exact}
    \ell(\theta;\xi)
    =
    1-\I_{
    \left\{
    \widehat{\cK}(\theta;\xi)=\mathcal{I}(\xi)
    \right\}},
\end{equation}
where $\mathcal{I}(\xi)$ and $\widehat{\cK}(\theta;\xi)$ denote the
ground-truth and recovered index sets, respectively. The parameter
vectors obtained from the noise-specific BO runs are subsequently
evaluated over the full noise sweep.

For Type~1 data, this evaluation revealed a substantial change in
performance in the high-noise regime. We therefore use a post-hoc
two-regime parameter selection: a low-noise parameter vector is used
for $\ep<\ep_{\mathrm{switch}}$ and a separately selected high-noise
parameter vector is used for $\ep\geq\ep_{\mathrm{switch}}$, where
\[
    \ep_{\mathrm{switch}}
    =
    \operatorname{logspace}(-2,0,20)_{13}
    \approx 0.1833.
\]
The high-noise parameters were selected from additional evaluations of
BO-generated candidates over the same noise sweep. Thus, this
two-regime selection is an empirical post-hoc tuning strategy rather
than an independently validated parameter-selection procedure.
For Type~2 data, a single parameter vector is used throughout the
noise sweep.

\subsubsection{Hand-gesture classification.}
For the experiments in Section~\ref{sec:HGexp}, we tune the same
parameter vectors over the same search ranges. The conditions include
the dataset (grayscale or binary), feature type (HOG or LBP), regularizer
($R_{p,1}$ or $R_{p,*}$), and target rank
$r\in\lceil n\{0.1,0.2,0.3,0.4,0.5\}\rceil$; for ADMM-P, we additionally
consider $p\in\{1,2,3,4\}$. Since the ground-truth index set is
unavailable, the per-trial loss is the SVM classification error
\begin{equation}\label{ch3:ch3-loss-svm}
    \ell(\theta;\xi)
    =
    1-\mathrm{Acc}_{\mathrm{SVM}}
    \left(\widehat{\cK}(\theta);\xi\right).
\end{equation}
For each candidate $\theta$, the SNMF algorithm is applied once to the
complete feature matrix to obtain $\widehat{\cK}(\theta)$, which is
fixed across $T=10$ stratified $80/20$ train/test splits. Each BO run uses $N_{BO}=40$ objective evaluations. Thus, feature selection is performed
before the train/test split, corresponding to a transductive feature-selection protocol.

\subsection{Ablation Studies}\label{ch3-sec:AblationStudies}
We now present a series of ablation studies that isolate the key design parameters of the ratio-of-norms algorithms: the power~$p$ in the numerator norm, the inner and outer iteration budgets~$J$ and~$K$, the matrix dimensions, and the regularizer choice. All results are averaged over the trial counts stated in each subsection. As in the main experiments, Algorithm~\ref{alg:RatioDCA} is evaluated at $p=1$, while
$p=1,2,3,4$ are considered for Algorithm~\ref{alg:Ratio-ADMMP}. Unless otherwise stated, all ablation experiments use the identity initialization $X_0=I_n$, consistent with the synthetic experiments in Section~\ref{sec:SynthExp}.

\subsubsection{Sensitivity of Power $p$}\label{ch3-sec:Abl1PowerP}

The ratio-of-norms model~\eqref{ch3:RatioModel} permits a generalized power~$p$ in the numerator, yielding the regularizer $R_p(X)$. In our experiments, Algorithm~\ref{alg:RatioDCA} (DCA) is implemented with $p=1$, whereas Algorithm~\ref{alg:Ratio-ADMMP} (ADMM-P) is evaluated for $p\in\{1,2,3,4\}$ through the $Z$-update. We sweep $p$ for ADMM-P across 15 noise levels in $[10^{-2},\,1]$, using the Type~1 ($50\times 55$, $r=10$) and Type~2 ($50\times 100$, $r=10$) datasets with 25 trials per noise level.

Under $R_{p,1}$ (Figure~\ref{fig:abl1-successRate}), ADMM-P maintains success rates near $1$ for $p\in\{2,3\}$ across the entire noise sweep, including the difficult transition band near $\ep \approx 0.2$--$0.3$. The powers $p=1$ and $p=4$ are more fragile in that band: each shows a sharp dip toward zero around $\ep \approx 0.2$ before recovering to near-perfect accuracy at $\ep = 1$. The choices $p=2,3$ are therefore more reliable under the $\ell_1$-based ratio.

\begin{figure}[h!]
\centering
\begin{tabular}{cc}
    \includegraphics[width=0.48\textwidth]{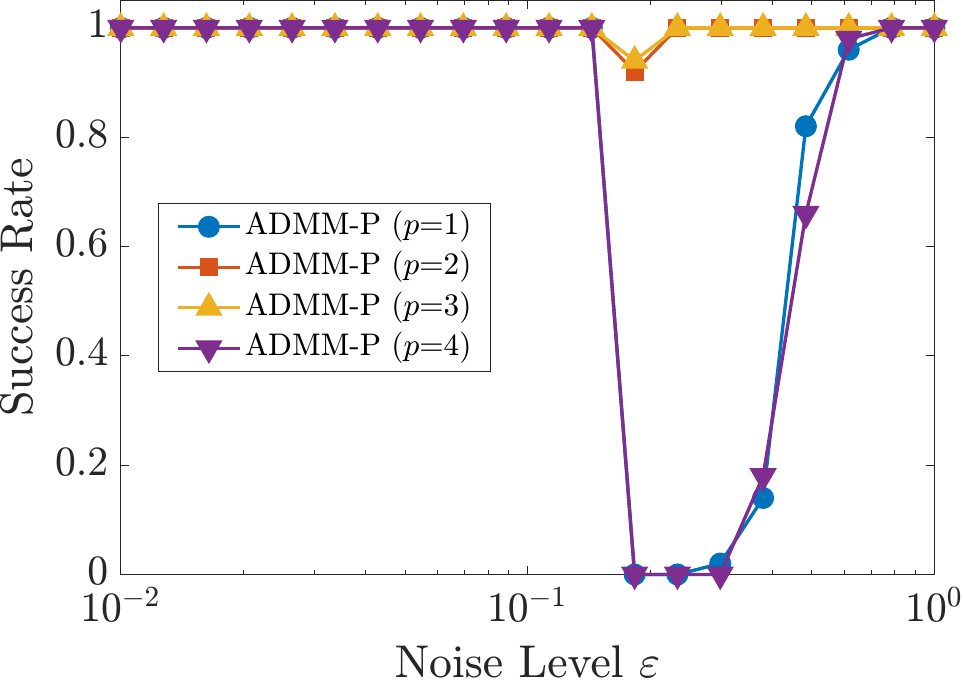}    &
    \includegraphics[width=0.48\textwidth]{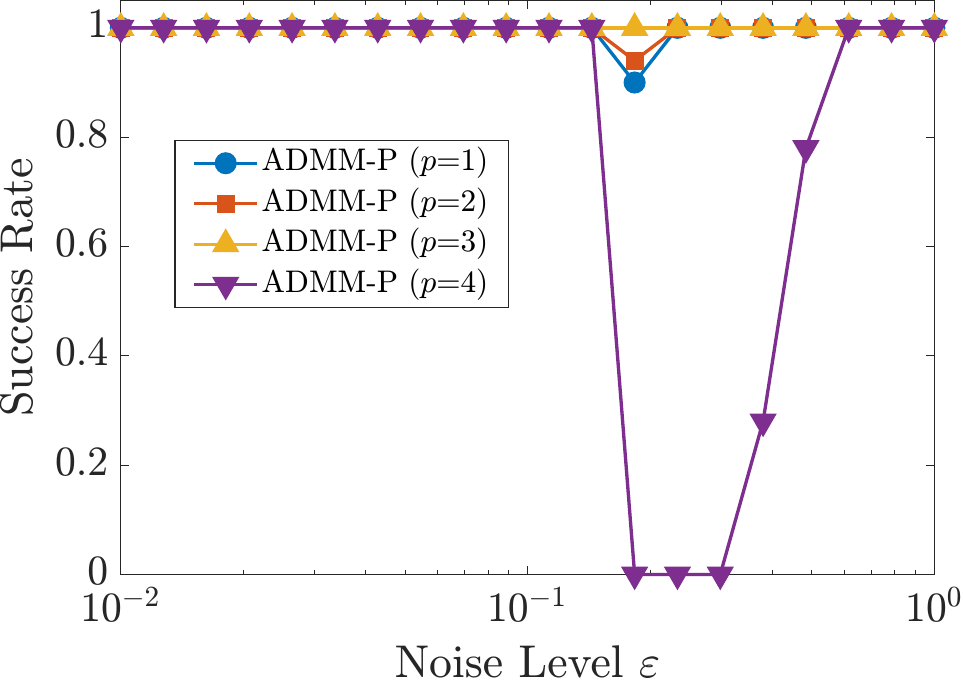}\\
    {(a) $R_{p,1}$} & {(b) $R_{p,*}$}
\end{tabular}
\caption{Success rate across $p\in\{1,2,3,4\}$ for Algorithm~\ref{alg:Ratio-ADMMP} on Type~1 data for the $R_{p,1}$ and $R_{p,*}$-regularized models.}
\label{fig:abl1-successRate}
\end{figure}

Under $R_{p,*}$ (Figure~\ref{fig:abl1-successRate}) the ranking inverts. Now $p\in\{1,2,3\}$ all maintain success rates near $1$ across the sweep, while $p=4$ exhibits the same fragile dip near $\ep\approx 0.2$--$0.3$ that the small powers showed under $R_{p,1}$.
The $p=4$ configuration shows reduced robustness under $R_{p,*}$, suggesting greater numerical sensitivity of the corresponding proximal update.
Increasing $p$ beyond $3$ provides no benefit under either regularizer and begins to hurt under $R_{p,*}$.

The optimal $p$ depends on the regularizer. Under $R_{p,1}$, $p\in\{2,3\}$ is the safe operating range; under $R_{p,*}$, the safe range is $p\in\{1,2,3\}$. Overall, $p=4$ is not recommended since it provides no clear accuracy benefit, and exhibits greater numerical sensitivity under the nuclear-norm proximal update.

\subsubsection{Inner and Outer Iteration Numbers}\label{ch3-sec:Abl2IterBudget}

Both proposed algorithms employ a nested loop structure: an outer loop (indexed by $k=0,\dotsc,K-1$) that updates the linearization point or the auxiliary variable $Y$, and an inner loop (indexed by $j=0,\dotsc,J-1$) that approximately solves the corresponding inner ADMM subproblem. The total computational budget is $\cO(K\!\cdot\!J)$ inner iterations, but the allocation between the two loops is a design choice. We investigate this trade-off by performing two sweeps at a fixed moderate noise level $\ep = 0.1$ with $p=2$:
\begin{enumerate}
    \item[(a)] \textbf{Inner sweep:} Fix $K=100$, vary $J\in\{1,\,2,\,5,\,10,\,20,\,50\}$.
    \item[(b)] \textbf{Outer sweep:} Fix $J=10$, vary $K\in\{5,\,10,\,25,\,50,\,100,\,200\}$.
\end{enumerate}
Each configuration is evaluated over 25 trials on both data types.

\begin{figure}[h!]
\centering
\begin{tabular}{cc}
    \centering
    \includegraphics[width=.48\textwidth]{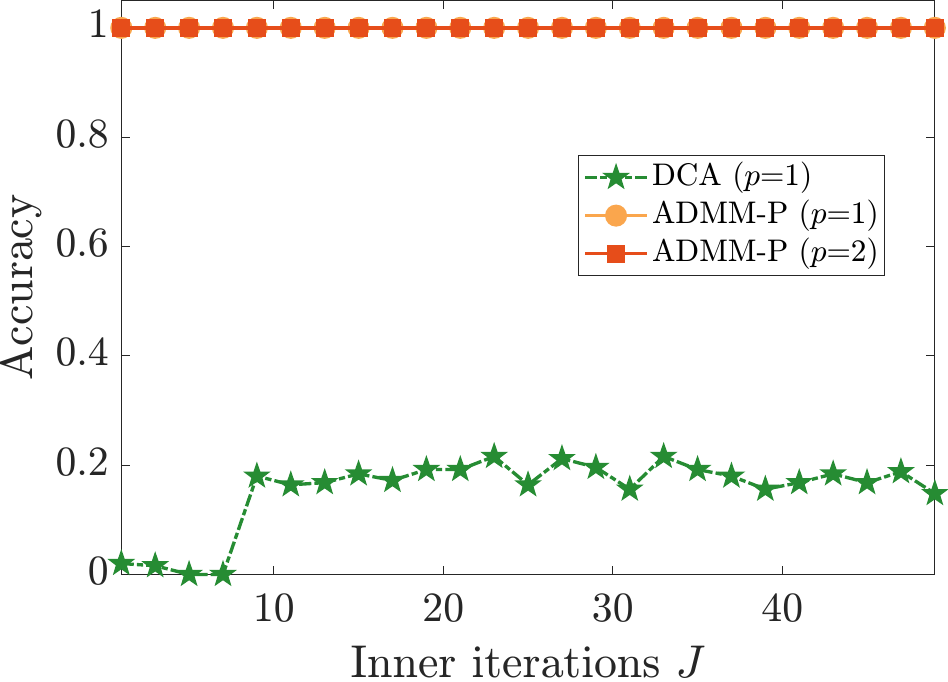}&
    \includegraphics[width=.48\textwidth]{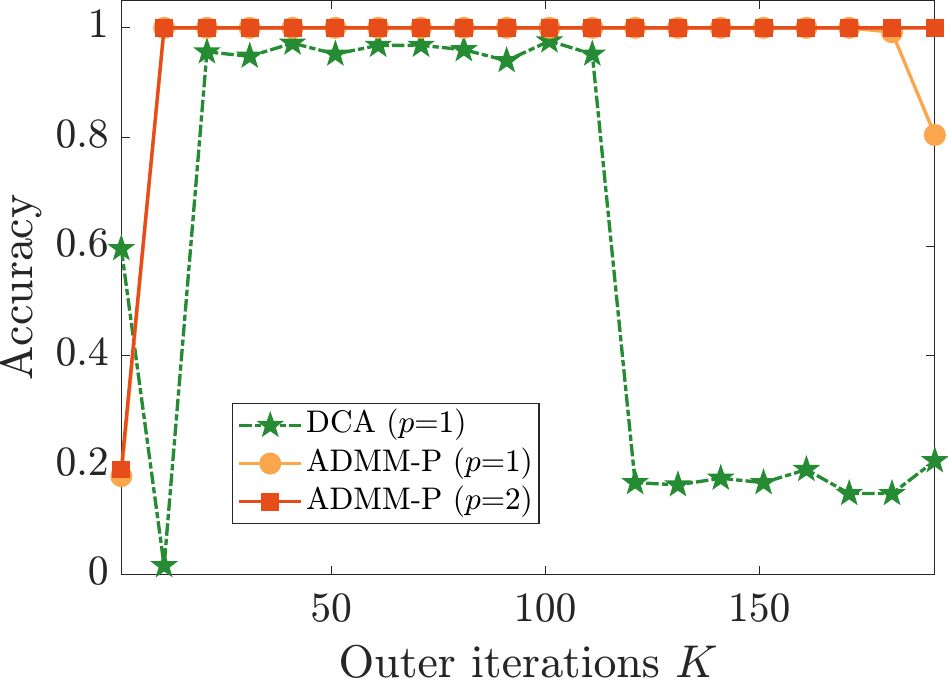}\\
    {(a) Inner sweep ($K=100$ fixed)}&
    {(b) Outer sweep ($J=10$ fixed)}
\end{tabular}
\caption{Accuracy versus iteration budget for the proposed $R_{p,1}$-regularized algorithms on Type~1 data.}
\label{fig:abl2-midpoint-L1-acc}
\end{figure}

ADMM-P is generally robust to the iteration budgets. Under $R_{p,1}$
(Figure~\ref{fig:abl2-midpoint-L1-acc}), it maintains near-perfect
accuracy across nearly all tested values of $J$ and $K$, including $J=1$
and $K=5$. This indicates that the ADMM consensus requires relatively
little inner work and only a few outer iterations to produce accurate
solutions. The main exception is ADMM-P with $p=1$, whose accuracy remains
stable through approximately $K=100$ before beginning to deteriorate near
$K=190$. In contrast, DCA is considerably more sensitive to the iteration
budgets. Its accuracy decreases after approximately $10$ inner iterations
and then plateaus near $0.2$. As $K$ varies, the accuracy oscillates
substantially during the first $30$ outer iterations, briefly stabilizes,
and then drops to a plateau near $0.2$.

\begin{figure}[h!]
\centering
\begin{tabular}{cc}
    \includegraphics[width=.48\textwidth]{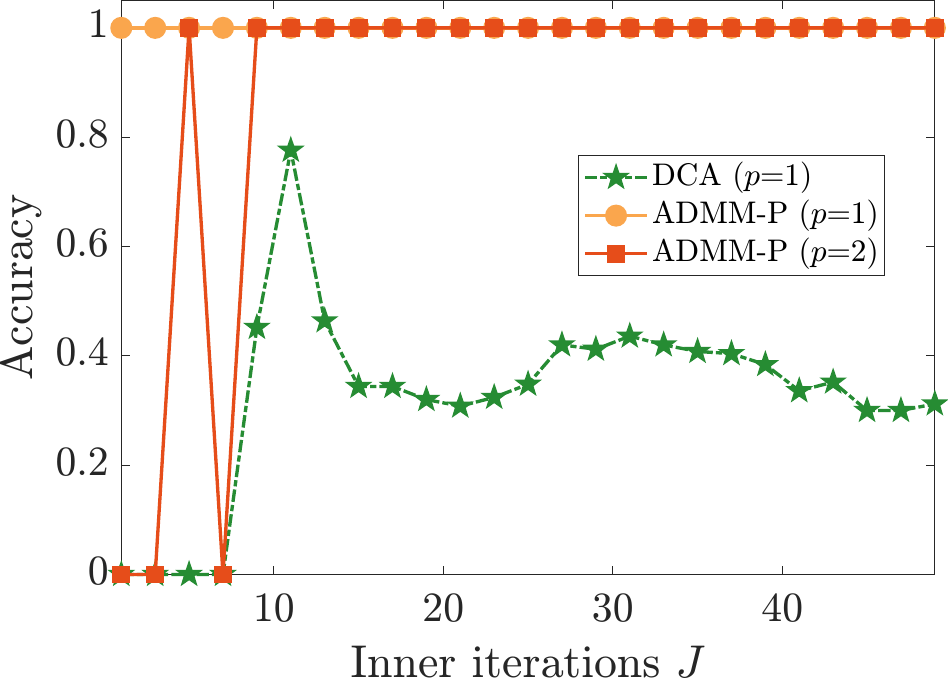}&
    \includegraphics[width=.48\textwidth]{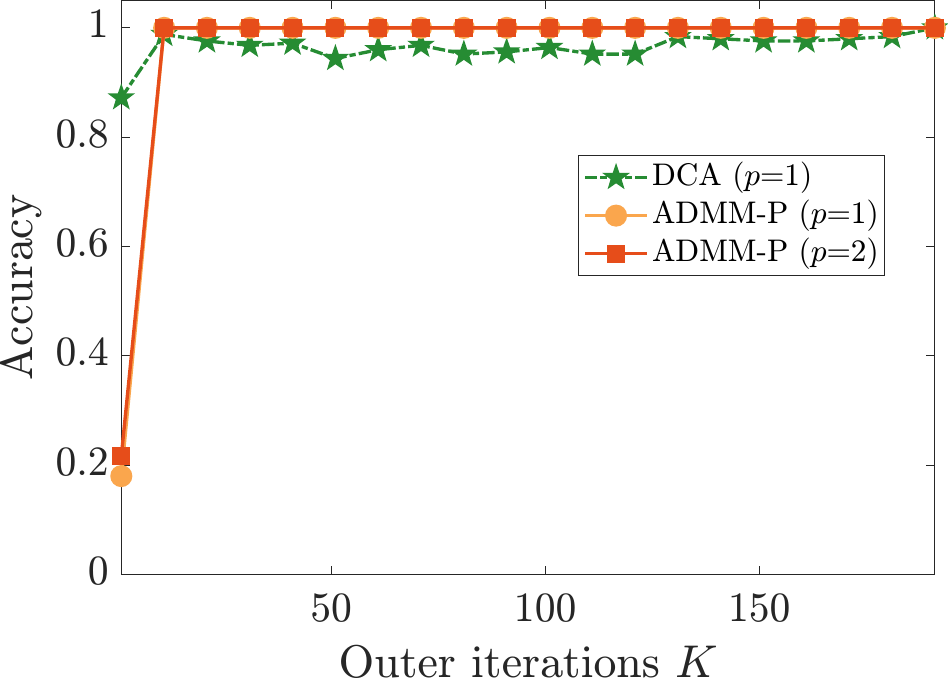}\\
    {(a) Inner sweep ($K=100$ fixed)}&
    {(b) Outer sweep ($J=10$ fixed)}
\end{tabular}
\caption{Accuracy versus iteration budget for the proposed $R_{p,*}$-regularized algorithms on Type~1 data.}
\label{fig:abl2-midpoint-nuclear-acc}
\end{figure}

Under $R_{p,*}$ (Figure~\ref{fig:abl2-midpoint-nuclear-acc}), the iteration budgets
have a more pronounced effect. ADMM-P remains robust with respect to the
outer iterations, reaching perfect accuracy after only a few iterations
and maintaining it thereafter. It is more sensitive to the inner budget,
particularly for $p=1$, where the accuracy oscillates sharply between
$0$ and $1$ for $J\leq 10$. DCA is relatively robust to the outer
iteration budget but remains sensitive to the inner budget. Its accuracy
exhibits a spike near $J=10$, followed by a decline and a plateau near
$0.4$. Overall, ADMM-P is substantially less sensitive to the allocation
of inner and outer iterations than DCA, although the nuclear-norm
regularizer introduces greater sensitivity to the inner iterations.

Overall, ADMM-P remains robust across both sweeps. In particular, the $p=2$ configuration is uniformly stable under both regularizers. ADMM-P is also largely insensitive to the iteration budgets on this problem: it tolerates $J$ as small as $1$ and $K$ as small as $5$, and remains stable across the tested configurations. In contrast, DCA requires more careful tuning, with a moderate inner budget ($J\approx 10$) and a small outer budget ($K\in[10,50]$), and its performance degrades when either iteration budget becomes too large. The conventional $(K,J)=(100,10)$ setting used in Section~\ref{sec:SynthExp} is therefore well suited to ADMM-P, while DCA appears more sensitive to this choice.

\subsubsection{Scalability across Matrix Dimensions}\label{ch3-sec:Abl3Scalability}

The purpose of this experiment is to evaluate how the proposed methods scale with problem size. We vary the dimensions of Type~2 synthetic data. Three configurations are tested for each data type:
\begin{itemize}
    \item \textbf{Small:} Type~1: $30\!\times\!35$, $r\!=\!5$; \quad Type~2: $30\!\times\!60$, $r\!=\!5$.
    \item \textbf{Medium:} Type~1: $50\!\times\!55$, $r\!=\!10$; \quad Type~2: $50\!\times\!100$, $r\!=\!10$.
    \item \textbf{Large:} Type~1: $100\!\times\!110$, $r\!=\!20$; \quad Type~2: $100\!\times\!200$, $r\!=\!20$.
\end{itemize}
We fix $p=2$, use 12 noise levels logarithmically spaced in $[10^{-2},\,1]$, and average the results over 20 trials per setting.

\begin{figure}[h!]
\centering
\begin{tabular}{cc}
    \includegraphics[width=.48\textwidth]{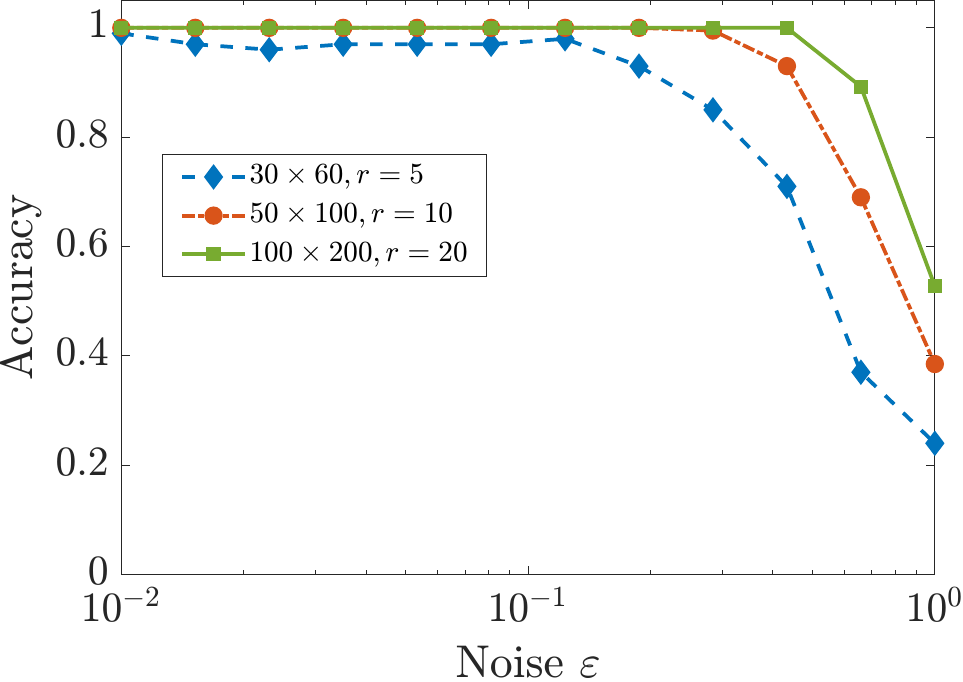} &
    \includegraphics[width=.48\textwidth]{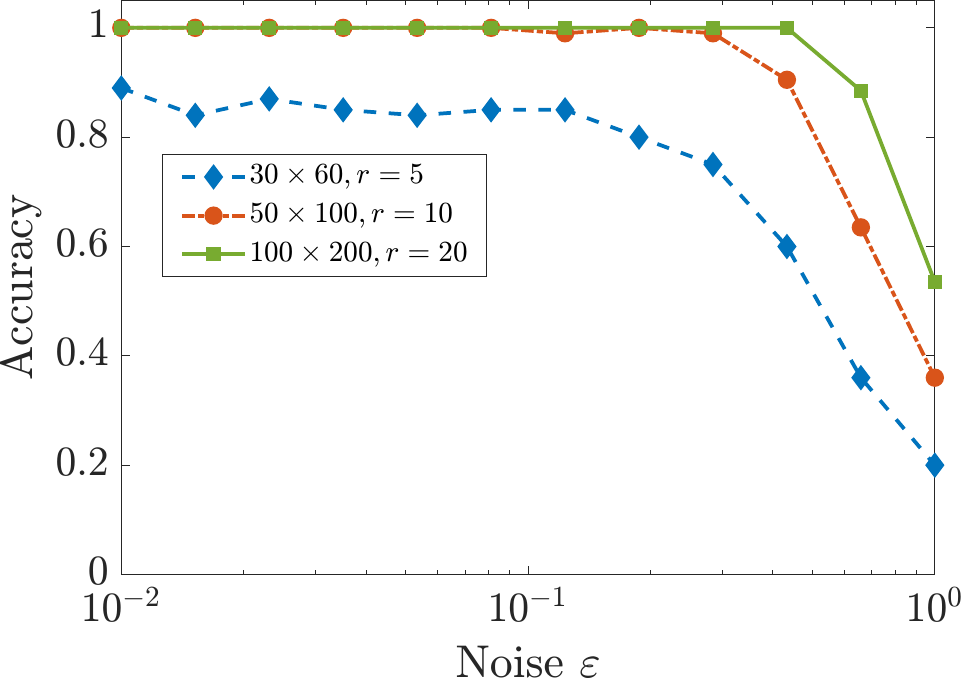}\\
    DCA & ADMM-P
\end{tabular}
\caption[Matrix size versus accuracy on Type~2 data]{
Accuracy of DCA and ADMM-P with $R_{p,1}$ regularization for different
matrix sizes on Type~2 data.}
\label{fig:abl3-dir_rand-L1-acc}
\end{figure}

\begin{figure}[h!]
\centering
\begin{tabular}{cc}
    \includegraphics[height=.21\textheight]{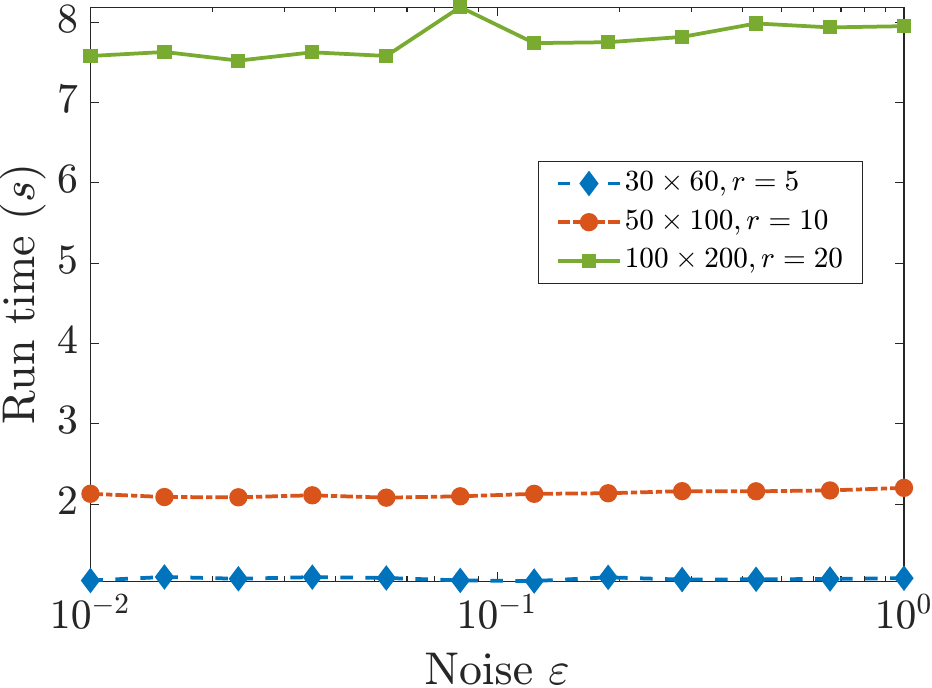} &
    \includegraphics[height=.21\textheight]{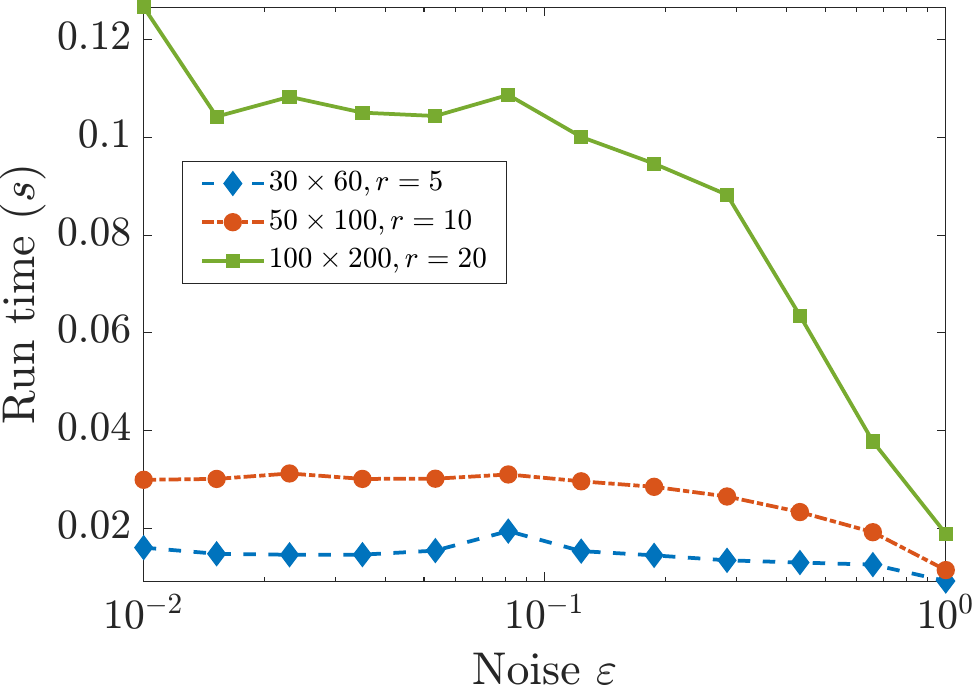}\\
    (a) $R_{p,*}$ & (b) $R_{p,1}$
\end{tabular}
\caption{Runtime scalability of ADMM-P with $R_{p,*}$ and $R_{p,1}$
regularization on Type~2 data. The nuclear-norm computation dominates
the runtime for $R_{p,*}$, whereas $R_{p,1}$ is substantially more
efficient as the matrix size increases.}
\label{fig:abl3-dir_rand-time}
\end{figure}

The accuracy results show a consistent dependence on problem size across
both data types and regularizers. For Type~2 data with $R_{p,1}$
regularization (Figure~\ref{fig:abl3-dir_rand-L1-acc}), the
$100\times200$ instances achieve the highest accuracy, followed by the
medium and $30\times60$ instances. The same general ordering is observed
with $R_{p,*}$. This behavior suggests a sample-size effect: as $n$
increases, the number of measurement rows $m$ also increases, providing
more information for identifying the columns indexed by $\cK$. Within
the recoverable noise regime, this additional information appears to
outweigh the increased dimensionality of the problem.

The runtime results in Figure~\ref{fig:abl3-dir_rand-time} are consistent
with the complexity analysis in Section~\ref{ch3-sec:CompComplexity}.
With $R_{p,1}$ regularization, ADMM-P remains computationally inexpensive
as the problem size increases. In contrast, the SVD required by
$R_{p,*}$ introduces an $\cO(n^3)$ computational cost and becomes the
dominant expense for larger matrices. Together, the accuracy and runtime
results indicate that ADMM-P scales favorably when paired with
$R_{p,1}$, providing a better balance between recovery accuracy and
computational cost for larger problems.

\subsection{Computational Complexity}\label{ch3-sec:CompComplexity}

We analyze the per-iteration cost of Algorithms~\ref{alg:RatioDCA} and~\ref{alg:Ratio-ADMMP}. Projecting an $n\times n$ matrix onto $\Om$ costs $\cO(n^2\log n)$ operations~\cite{Gillis2018}. For each algorithm, the dominant cost depends on the regularizer: $R_{p,1}$ admits cheap proximal updates of $\cO(n^2)$, while $R_{p,*}$ requires an SVD at $\cO(n^3)$. In both cases, the projection contributes an additional $\cO(n^2\log n)$.

For Algorithm~\ref{alg:RatioDCA}, the matrix in the $X$-subproblem is fixed and can be Cholesky-factorized once at a cost of $\cO(n^3)$. For the projected-proximal implementation used in the numerical experiments, each inner iteration requires an $X$-update via forward and backward substitutions costing
$\cO(n^2)$, a $V$-update costing $\cO(n^2)$ or $\cO(n^3)$ depending on the norm, a projection of $V$ onto $\Om$ costing $\cO(n^2\log n)$, and a $Z$-update costing $\cO(n^2)$. Computing the outer $\alpha$-update costs $\cO(n^2)$ for $R_{p,1}$ and $\cO(n^3)$ for $R_{p,*}$ due to the SVD. For Algorithm~\ref{alg:Ratio-ADMMP}, the $X$- and $Z$-updates are unprojected and cost $\cO(n^2)$ or $\cO(n^3)$, while the additional $W$-update absorbs the $\cO(n^2\log n)$ projection. Thus, for the implementations used in the numerical experiments, the per-inner-iteration cost is $\cO(n^2\log n)$ under $R_{p,1}$ and $\cO(n^3)$ under $R_{p,*}$. Table~\ref{tab:ComputationalComplexity} summarizes these results.

\begin{table}[htbp]
\centering
\caption{Per-inner-iteration computational complexities for Algorithms~\ref{alg:RatioDCA}--\ref{alg:Ratio-ADMMP} under each regularizer choice.}
\label{tab:ComputationalComplexity}
\begin{tabular}{lcc}
\hline \textbf{Algorithm / Operation} & \textbf{Complexity: $R_{p,1}(X)$} & \textbf{Complexity: $R_{p,*}(X)$} \\ \cmidrule(lr){1-1}\cmidrule(lr){2-2}\cmidrule(lr){3-3} $\cP_{\Om}(\cdot)$ & $\cO(n^2\log n)$ & $\cO(n^2\log n)$ \\
Alg.~\ref{alg:RatioDCA} or Alg.~\ref{alg:Ratio-ADMMP} & $\cO(n^2\log n)$ & $\cO(n^3)$ \\ \hline
\end{tabular}
\end{table}

Both proposed algorithms share the same asymptotic per-inner-iteration cost under each regularizer. Thus, the wall-clock differences observed in Section~\ref{ch3-sec:Abl3Scalability} primarily reflect constant factors and iteration counts rather than differences in asymptotic complexity. The observed runtime gap between the two regularizers at larger problem sizes is consistent with the higher asymptotic cost of
the SVD-based $R_{p,*}$ update, further demonstrating that $R_{p,1}$ is the more computationally efficient choice at scale.

\section{Conclusions}\label{sec:con}

In this work, we introduced powered ratio-of-norms regularization for separable nonnegative matrix factorization using entrywise $\ell_1$- and nuclear-norm regularizers, and developed DCA- and ADMM-based algorithms for the resulting nonconvex models. Numerical experiments on synthetic and hand gesture datasets demonstrate the effectiveness of the proposed methods for basis identification and feature selection. In particular, the ADMM-based methods generally provide more reliable recovery than DCA, while the entrywise $\ell_1$-based formulation offers a favorable balance between accuracy and computational cost. The experiments also demonstrate the influence of the power parameter $p$ on recovery performance and robustness. The current analysis relies on standard exact-subproblem assumptions, while the numerical performance remains sensitive to parameter selection in challenging noise regimes. Future work will include convergence analysis for inexact implementations,
scalable implementations, adaptive parameter selection, and extensions to other structured matrix and tensor factorization problems.

\bibliographystyle{unsrt}
\bibliography{ref}

\end{document}